\RequirePackage{fix-cm}
\documentclass[smallcondensed]{svjour3}     
\smartqed  
\usepackage{hyperref}
\usepackage{graphicx}

\usepackage{mathptmx}      
\usepackage{lipsum}
\usepackage{amsfonts}
\usepackage{graphicx}
\usepackage{epstopdf}
\usepackage{algorithmic}
\usepackage{tikz}
\usepackage{csquotes}
\usepackage{circuitikz}
\usepackage{siunitx}
\usepackage{amsmath}
\usepackage{amsthm}
\usepackage[numbers]{natbib}
\ifpdf
  \DeclareGraphicsExtensions{.eps,.pdf,.png,.jpg}
\else
  \DeclareGraphicsExtensions{.eps}
\fi

\definecolor{Pratik}{rgb}{1, 0, 0}

\theoremstyle{definition}
\newtheorem{defn}{Definition}[section]
\begin{document}

\title{Asymptotic Analysis for Overlap in Waveform Relaxation Methods for RC Type Circuits
}

\titlerunning{Asymptotic Analysis for Overlap in WR methods for RC  Circuits}        

\author{Martin J. Gander         \and
        Pratik M. Kumbhar*   \and    Albert E. Ruehli}


\institute{Martin J. Gander \at
              Section de Math\'ematiques, University of Geneva, Geneva CH-1211, Switzerland.\\
               \email{martin.gander@unige.ch}           
           \and
           Pratik M. Kumbhar \at
              Section de Math\'ematiques, University of Geneva, Geneva CH-1211, Switzerland.\\
               \email{pratik.kumbhar@unige.ch}
           \and
           Albert E. Ruehli \at    
                EMC Laboratory, Missouri University of Science And Technology, U.S. \\
                \email{ruehlia@mst.edu}
}

\date{Received: date / Accepted: date}

\maketitle

\begin{abstract}
  Waveform relaxation (WR) methods are based on partitioning large
  circuits into sub-circuits which then are solved separately for
  multiple time steps in so called time windows, and an iteration
    is used to converge to the global circuit solution in each time
    window. Classical WR converges quite slowly, especially when
  long time windows are used. To overcome this issue, optimized WR
  (OWR) was introduced which is based on optimized transmission
  conditions that transfer information between the sub-circuits
  more efficiently than classical WR. We study here for the first
    time the influence of overlapping sub-circuits in both WR and OWR
  applied to RC circuits. We give a circuit interpretation of the
    new transmission conditions in OWR, and derive closed form
  asymptotic expressions for the circuit elements representing the
    optimization parameter in OWR. Our analysis shows that the
    parameter is quite different in the overlapping case, compared to
    the non-overlapping one. We then show numerically that our
    optimized choice performs well, also for cases not covered by
      our analysis. This paper provides a general methodology to
      derive optimized parameters and can be extended to other circuits or
      system of differential equations or space-time PDEs.
      
\end{abstract}

\section{Introduction}\label{intro}
   Electric circuits play a crucial role in our everyday life:
  they can be found in computers, cell phones, chargers, but also in
  cars, watches, and more and more household appliances like stoves,
  fridges and so on. Before buildings these circuits, they are
  usually simulated using computer programs.  Simulating
 circuits means developing a mathematical model to replicate the
 behavior of the real circuit, and then solve this model
  numerically. These simulations help the circuit designers to test
 their ideas, and optimize circuit parameters to achieve the
 desired output. Simulations save a lot of time and cost when
  designing circuits, and one can minimize the risk of unwanted 
 hazards. Further, the number of electronic devices is continuously increasing 
 tremendously which in turn increases the need for new design tools and techniques.
 Circuit solvers like the many different SPICE are no exception. Not only new approaches
 are desired but the tools need to
 have  an ever increasing capacity to solve larger problems.

 The specific class of problems considered in this work is a subset.  We consider a
  sub-class of circuits which are called RC circuits which are composed of resistors (R)
  and capacitors (C). Of course, the circuits will also include current and
  voltage sources. This circuit is also called RC filter since it is
 used to filter certain frequencies and let pass others. The most
 common are high pass filters and low pass filters. We study
  here the low pass RC filter circuit which is shown in
Fig. \ref{RCinf}.
\begin{figure}
  \centering
  \begin{circuitikz}[scale=0.47, every node/.style={scale=0.7}]
   \draw (2,1) to [short] node [ground] {} (2,1);
   \draw (2,1) to [american current source] (2,4);
   \draw (1,2.1) node[above]{\large{$I_{s}$}};
   \draw  (2,4) to (6.5,4);
   \draw (3.5,4) to [R]  (3.5,1);
   \draw (4.2,2.1) node[above]{\large{$R_{s}$}};
   \draw (3.5,1) to [short] node [ground] {} (3.5,1);
   \draw  (4,4) to (6,4);
   \draw[black,fill=black] (5.55,4) circle (.7ex);
   \draw (5.5,4) node[above] {\large{$v_{-\infty}$}};
   \draw (5.5,4) to [C=\large{\(C_{-\infty}\)}]  (5.5,0.5);
   \draw (5.5,1) to [short] node [ground] {} (5.5,1);
   \draw  (5.5,4) to [R=\large{\(R_{-\infty}\)}] (9,4);
   \draw[black,fill=black] (9,4) circle (.7ex);
   \draw (9.5,4) node[above] {\large{$v_{-\infty+1}$}};
   \draw (9,4) to [C=\large{\(C_{-\infty+1}\)}]  (9,1);
   \draw (9,1) to [short] node [ground] {} (9,1);
   \draw  (9,4) to (10,4);
   \draw [dashed] (10,4) to (13,4);
   \draw  (13,4) to (14,4);
   \draw[black,fill=black] (14,4) circle (.7ex);
   \draw (14,4) node[above] {\large{$v_{-1}$}};
   \draw (14,4) to [C=\large{\(C_{-1}\)}]  (14,1);
   \draw (14,1) to [short] node [ground] {} (14,1);
   \draw  (14,4) to [R=\large{\(R_{-1}\)}] (18,4);
   \draw[black,fill=black] (18,4) circle (.7ex);
   \draw (18,4) node[above] {\large{$v_{0}$}};
   \draw (18,4) to [C=\large{\(C_{0}\)}]  (18,1);
   \draw (18,1) to [short] node [ground] {} (18,1);
   \draw  (18,4) to [R=\large{\(R_{0}\)}] (22,4);
   \draw[black,fill=black] (22,4) circle (.7ex);
   \draw (22,4) node[above] {\large{$v_{1}$}};
   \draw (22,4) to [C=\large{\(C_{1}\)}]  (22,1);
   \draw (22,1) to [short] node [ground] {} (22,1);
   \draw (22,4) to [R=\large{\(R_{1}\)}] (25,4);
   \draw [dashed] (25,4) to (27,4);
   \end{circuitikz}
   \caption{RC Circuit of infinite length.}
    \label{RCinf}
\end{figure}
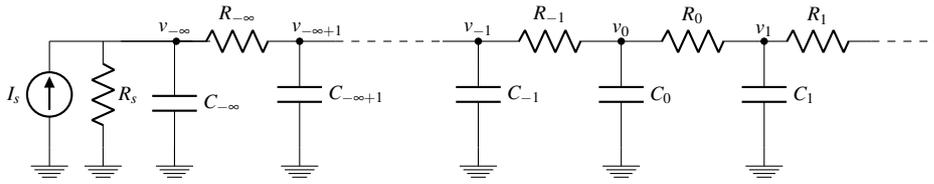
    This circuit allows signals with frequencies lower than a certain cut
 off frequency to pass. Such circuits are often used in acoustics,
 optics and electronics and hence their study is of great
 importance. We formulate an effective mathematical model for its
 simulation using the well known and widely used MNA (Modified
 Nodal Analysis), which resembles the finite element assembly
  procedure in  numerical partial differential equations (PDEs) \cite{FEM_book,Griffiths1986}. This
  leads to a large system of ordinary differential equations
 (ODEs) where the unknowns are voltages at the nodes \cite{MNA}. It is
 computational costly to solve such systems especially if their
 size is in the order of millions. Also, since 2004, the CPU
 frequency has been stagnating around 3 GHz and there is little
 hope that in the near future there will be any improvement \cite{Nataf}. Due
 to the availability of large numbers of processors, we are left with
 the option of using parallel computing in order to increase
 simulation speed. Over the years, various numerical tools
 have been developed to solve PDEs and large systems of equations
 effectively using parallel computing. One such tool is Waveform
  Relaxation.
  
      Waveform Relaxation (WR) was developed in 1982 by Lelarasmee,
    Ruehli and Sangiovanni-Vincentelli \cite{WR_LRS}, for analyzing
    electric circuits. In WR, the circuit or the system of
    differential equations is decomposed into many sub-circuits or
    subsystems which are then solved independently over an entire
    time window using standard time integration techniques. After each
    solve, information is transferred between neighboring
    subsystems using transmission conditions and the process is
    repeated until convergence is reached. The classical WR
    algorithm converges slowly when large time windows are used. To
    overcome this issue, smart transmission conditions were
    introduced, which led to optimized waveform relaxation (OWR)
    since different parameters in the transmission conditions need to be
    optimized. The convergence analysis of the classical WR methods
    for RC circuits was carried out by Ruehli et al \cite{WR_RC_Convergence}.
     Further, studies of WR and OWR algorithms applied to
    small RC circuits can be found in \cite{RC_small_KG, RC_small_GA},
   and for large RC circuits, see \cite{RC_KG,RC_GK24}. There are
   also analyses at the discrete level for OWR applied to RC circuits
   \cite{RC_discrete_SM}. WR and OWR were also studied for RLCG
  transmission lines \cite{RLCG_KG_OWR_transmissionlines1,
  RLCG_KG_WR_transmissionlines,RLCG_KG_OWR_transmissionlines,RLCG_GK25,
  RLCG_GA_OWR_line}. A different approach for the interconnect
  circuits has been studied in \cite{coupling_WR_tarik}, where the
  interconnect circuit is partitioned using Norton interfaces which
  depend upon resistors.
   However, in all these references, both WR and OWR were studied
   for the case of non-overlapping circuits. We show here that
   WR for circuits is related to domain decomposition for PDEs, in the
   case of the RC circuit the heat equation \cite{heat_GS, heat_GL},
   and for RLCG circuits the wave equation \cite{wave_GLF_optimal1d},
   and domain decomposition methods often use overlap to improve
   convergence. In this article, we study the effect of overlap on the
   convergence factor of WR and OWR and present for the first time
   an analysis for the optimized parameter in the transmission
   conditions for RC circuits of infinite length, a result which was
  announced in the short proceedings paper \cite{RC_GK24}.
  
   In WR methods, only the space domain or system of differential equations is split
  into multiple sub-domains or subsystems. One can further combine these WR methods
  with certain other time parallel methods, where the time domain is divided into multiple
  time subdomains. On each time subdomain of the WR method, Parareal
  \cite{time_parareal,time_parareal2} or 
  the pipeline Schwarz waveform relaxation method \cite{time_pipeline}, or  Waveform Relaxation 
  with Adaptive Pipelining (WRAP) \cite{time_WRAP},
  or Revisionist Integral Deferred Correction (RIDC) \cite{RIDC_ong,time_RIDC} can be considered. In the 
  pipeline  Schwarz waveform relaxation method and the WRAP method,
  WR iterates are  computed in a pipeline manner to provide speedup and stability, while RIDC
  produces high order solutions in the same time as the lower order methods \cite{time_RIDC}.
  Further, multigrid WR methods \cite{multigridWR,multigrid_heat} can also be implemented 
  for parabolic differential equations.
  
   WR methods are becoming popular for the simulation of field-circuit coupled 
  problems. Such problems arises if one wants to look closely into the device. In this
  situation, a lumped circuit element is replaced by a PDE model. Another perspective
  for such problems comes from the practical use by the engineer, where different parts of a domain
  are not equally important and hence some of its properties can be neglected
  by replacing a part of it by a simple electric circuit, while in other parts the space-time
  discretization of the PDE is kept. Recently WR method were used for such field-circuit problems at
  CERN to
  simulate the quench protection system \cite{FC_steam,GS_FC}. Moreover, if the RC circuits in
  the field-circuit coupled problem is divided into multiple sub-circuits, then
  the results obtained in our manuscript can be directly implemented in them. Further,
  OWR methods are also applied for the coupled electromagnetic field and power-electronic
  simulations \cite{WR_FC_Cosimulation}.
  
    We explain in Section \ref{Sec:MathModel} how one can obtain the
  mathematical model for the RC circuit of infinite length using
  MNA, and we show that the RC circuit is an approximation to the time
  dependent heat equation. In Section \ref{Sec:WRAlgo} we present and
  study the classical WR algorithm with overlap for RC circuits, and
  identify its convergence problems for low frequencies. In Section
  \ref{Sec:OWR} we introduce the OWR with overlap for RC circuits.  We
  give a circuit interpretation of the new transmission conditions,
  and also show that they can be interpreted as Robin transmission
  conditions for the associated heat equation. We then optimize the
  transmission conditions of OWR in Section \ref{Sec:Optimization}
  using asymptotic techniques. In Section \ref{Sec:ManySubCircuits} we
  show how to generalize our results to the many sub-circuit case.
We perform some numerical tests to support our theoretical results
in Section \ref{Sec:Numerical}, and present our conclusions in
  Section \ref{Sec:Conclusion}.
  
\section{Mathematical Model}\label{Sec:MathModel}

We are interested in determining the voltages at the nodes of an RC
 circuit as in Fig. \ref{RCinf}. To develop a mathematical model
 of this circuit, we use MNA, originally described by Ho,
  Ruehli and Brennan \cite{MNA} in 1975. This technique produces a
 system of ODEs of the form
 $\dot{\textbf{v}}=A\textbf{v}+\textbf{f}$, where the entries of
 the matrix $A$ contain the elements of the circuit, $\textbf{v}$
 is the unknown vector of voltages at the nodes, and
 $\textbf{f}$ contains the source terms. We now show how the
 matrix $A$ is built using MNA, which is similar to the
 finite element assembly procedure \cite{book_FEM}.

In an electric circuit, any electric device with two terminals is
 called an element, and its terminals are called nodes. An
  electrical circuit is thus a system consisting of a set of elements
 and a set of nodes. The contribution of every element to the matrix
 equation is described by means of a template, called an element stamp
 \cite{book_circuit}. Each element has a unique stamp which is given
 with the help of Kirchoff's Current Law (KCL) and Kirchoff's Voltage
 Law (KVL). For example, consider a resistor $R$ placed between nodes
 with voltages $v^{+}$ and $v^{-}$, as shown in Fig. \ref{stamp_R}
 (left).
\begin{figure}
\centering
\begin{minipage}{0.45\textwidth}
\centering
 \begin{circuitikz}[scale=0.6, every node/.style={scale=0.7}]
   \draw (3,2) to [R]  (7,2);
   \draw (5,2.5) node[above]{\large{$R$}};
   \draw (3,2.2) node[above]{\large{$v^{+}$}};
   \draw (7,2.2) node[above]{\large{$v^{-}$}};
   \draw (3.5,2) to [short,i=$ $] (4,2);
   \draw (3.75,1.9) node[below]{\large{$i$}};
   \draw[black,fill=black] (3,2) circle (.5ex);
   \draw[black,fill=black] (7,2) circle (.5ex);
\end{circuitikz}
\end{minipage}
\begin{minipage}{0.45\textwidth}
\begin{tabular}{ c c c c c } 
\hline
 & $v^{+}$ & $v^{-}$ & RHS\\
\hline
$v^{+}$  & $1/R$ & $-1/R$ & \\ 
$v^{-}$& $-1/R$ & $1/R$ & \\ 
\hline
\end{tabular}
\end{minipage}
\caption{Element stamp for a resistor $R$.}
\label{stamp_R}
\end{figure}
   Let $i$ be the current passing through it. Then using KCL, we have
    $i=v^{+}\left(\frac{1}{R}\right)-v^{-}\left(\frac{1}{R}\right)$. The
    corresponding element stamp is shown in Fig. \ref{stamp_R}
     (right). The element stamp for a capacitor $C$ is shown in Fig. \ref{stamp_C}.
\begin{figure}
\centering
\begin{minipage}{0.45\textwidth}
\centering
 \begin{circuitikz}[scale=0.6, every node/.style={scale=0.7}]
   \draw (3,2) to [C]  (7,2);
   \draw (5,2.5) node[above]{\large{$C$}};
   \draw (3,2.2) node[above]{\large{$v^{+}$}};
   \draw (7,2.2) node[above]{\large{$v^{-}$}};
   \draw[black,fill=black] (3,2) circle (.5ex);
  \draw[black,fill=black] (7,2) circle (.5ex);
\end{circuitikz}
\end{minipage}
\begin{minipage}{0.45\textwidth}
\begin{tabular}{ c c c c c  } 
\hline
 & $v^{+}$ & $v^{-}$ & RHS \\
\hline
$v^{+}$  & $C\frac{d}{dt}$ & -$C\frac{d}{dt}$  & \\ 
$v^{-}$& $-C\frac{d}{dt}$ & $C\frac{d}{dt}$ & \\ 
\hline
\end{tabular}
\end{minipage}
\caption{Element stamp for a capacitor $C$.}
\label{stamp_C}
\end{figure}

To build the "descriptor system" form
$M\dot{\textbf{v}}=K\textbf{v}+\tilde{\textbf{f}}$,
 we start with a zero RHS vector $\tilde{\textbf{f}}\in
 \mathbb{R}^N$, and zero matrices $M,K\in \mathbb{R}^{N\times N}$,
  where $N$ is the number of nodes in the circuit. The vector $\textbf{v}$ contains the unknown voltages at the nodes 
  and remains unchanged throughout
 this process. We then read every element
 of the circuit one by one. As an element is read, its element
 stamps are added to the matrices $M$ and $K$. Since the element stamp
 of a capacitor contains derivative terms, its element
  stamp is added to the matrix $M$, while the element stamp
   of a resistor is added in matrix the $K$. The independent voltage source 
 term is stamped into the vector $\tilde{\textbf{f}}$, see Fig. \ref{stamp_v}
 \begin{figure}
\centering 
\begin{minipage}{0.45\textwidth}
\centering
 \begin{circuitikz}[scale=0.6, every node/.style={scale=0.7}]
   \draw (3,2) to [american current source] (7,2);
   \draw (5,2.5) node[above]{\large{$i_{s}=I$}};
      \draw (3,2.2) node[above]{\large{$v^{+}$}};
   \draw (7,2.2) node[above]{\large{$v^{-}$}};
   \draw[black,fill=black] (3,2) circle (.5ex);
  \draw[black,fill=black] (7,2) circle (.5ex);
\end{circuitikz}
\end{minipage}
\begin{minipage}{0.45\textwidth}
\begin{tabular}{ c c c c c c } 
\hline
 & $v^{+}$ & $v^{-}$ & RHS \\
\hline
$v^{+}$  &    &    & -I  \\ 
$v^{-}$  &    &    & +I    \\
\hline
\end{tabular}
\end{minipage}
\caption{Element stamp for an independent current source.}
\label{stamp_v}
\end{figure}
 for the element stamp of an independent voltage source.
 This process continues until all elements of the circuit 
 are read and thus we obtain a system of differential
  equations in time of the form 
 $\dot{\textbf{v}}=A\textbf{v}+\textbf{f}$, where 
 $A:=M^{-1}K$ and $\textbf{f}:=M^{-1}\tilde{\textbf{f}}$.
 For the infinite RC circuit in Fig. \ref{RCinf},
 this system looks like
\begin{equation}\label{ce1}
   \frac{d\textbf{v}}{dt}=\begin{bmatrix}
                  \ddots & \ddots  & \ddots &      &   \\
                         & a_{-1}  & b_{0} & c_{0} &   & \\
                         &         & a_{0} & b_{1} & c_{1} &  & \\
                         &         &       & a_{1} & b_{2} & c_{2} &  & \\

                         &     &   &   & \ddots&\ddots&\ddots\\ 
                 \end{bmatrix} \textbf{v} + \textbf{f},
  \end{equation}
where 
\[
 a_{i}=\frac{1}{R_{i}C_{i+1}},\quad b_{i}=-\left(\frac{1}{R_{i-1}}+\frac{1}{R_{i}}\right)\frac{1}{C_{i}}, \quad  c_{i}=\frac{1}{R_{i}C_{i}}, \quad
i \in \mathbb{Z}, 
 \]
and $\textbf{f}=(I_{s}(t)/C_{-\infty}, 0, 0, . . . ,0)^{T}$.
 This system of ODEs can also be viewed as semi-discretization by the
  method of lines of a heat equation in space-time: if we consider
  small resistors and capacitors, $R_i\approx \Delta x$ and
  $C_i\approx \Delta x$, then each equation of the system \eqref{ce1}
  takes the form
\begin{equation}\label{HeatEquationDiscrete}
  \frac{dv_{i}}{dt}= \frac{v_{i-1}-2v_{i}+v_{i+1}}{\Delta x^2}+f_i,
\end{equation}
 and as $\Delta x\rightarrow 0$, we arrive at heat equation
\begin{equation}\label{HeatEquation}
  \frac{\partial v}{\partial t}=\frac{\partial^2 v}{\partial x^2}+f.
\end{equation}
 Hence we can consider the RC circuit of infinite length as an
 approximation for the one dimensional heat equation on the
 unbounded domain $\Omega_{x}=(-\infty ,\infty)$.

 For a short vector $\textbf{v}$, one can solve the
  system of ODEs \eqref{ce1} by simple discretization and matrix
 manipulation. However if the size of the vector $\textbf{v}$
  is in the millions and for more complicated circuits, solving the
  system \eqref{ce1} is computationally costly, and one needs
 parallel methods. One such method is WR as we now show.
  
\section{The WR Algorithm}\label{Sec:WRAlgo}

For a general WR algorithm, one decomposes the system \eqref{ce1}
 into many sub-systems, but to understand the key features of
 WR, we consider a decomposition into two sub-systems only.
  We decompose the infinite RC circuit from Fig. \ref{RCinf} 
  at node 0 into two equal sub-circuits, and motivated by the
   PDE relation shown in \eqref{HeatEquation} and the Schwarz WR algorithms  studied in  \cite{heat_GS,heat_GL} which use overlap,
  we include an overlap of $n$ nodes in the first sub-circuit.
   Let us denote the first sub-system unknowns by $\textbf{u(t)}$ 
   and the second sub-system unknowns by $\textbf{w(t)}$,
    $\textbf{u(t)}:=(\ldots,u_{-1},u_0,\ldots,u_{n})^T=(\ldots,v_{-1},v_0,\ldots,v_{n})^T $ and $\textbf{w(t)}
:=(w_{1},w_2,\ldots)^T=(v_{1},v_2,\ldots)^T$. Fig. \ref{WR_circuit}
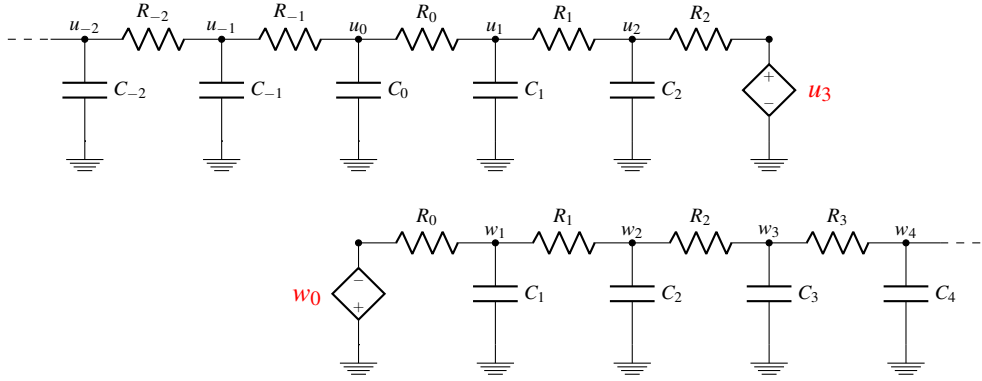
\begin{figure}[t]
  \begin{circuitikz}[scale=0.45, every node/.style={scale=0.7}]
   \draw [dashed] (-1.25,4) to (0,4);
   \draw  (0,4) to (1,4);
   \draw[black,fill=black] (1,4) circle (.7ex);
   \draw (1,4) node[above] {\large{$u_{-2}$}};
   \draw (1,4) to [C=\large{\(C_{-2}\)}]  (1,1);
   \draw (1,1) to [short] node [ground] {} (1,1);
   \draw  (1,4) to [R=\large{\(R_{-2}\)}] (5,4);
   \draw[black,fill=black] (5,4) circle (.7ex);
   \draw (5,4) node[above] {\large{$u_{-1}$}};
   \draw (5,4) to [C=\large{\(C_{-1}\)}]  (5,1);
   \draw (5,1) to [short] node [ground] {} (5,1);
   \draw  (5,4) to [R=\large{\(R_{-1}\)}] (9,4);
   \draw[black,fill=black] (9,4) circle (.7ex);
   \draw (9,4) node[above] {\large{$u_{0}$}};
   \draw (9,4) to [C=\large{\(C_{0}\)}]  (9,1);
   \draw (9,1) to [short] node [ground] {} (9,1);
   \draw  (9,4) to [R=\large{\(R_{0}\)}] (13,4);
   \draw[black,fill=black] (13,4) circle (.7ex);
   \draw (13,4) node[above] {\large{$u_{1}$}};
   \draw (13,4) to [C=\large{\(C_{1}\)}]  (13,1);
   \draw (13,1) to [short] node [ground] {} (13,1);
   \draw (13,4) to [R=\large{\(R_{1}\)}] (17,4);
   \draw[black,fill=black] (17,4) circle (.7ex);
   \draw (17,4) node[above] {\large{$u_{2}$}};
   \draw (17,4) to [C=\large{\(C_{2}\)}]  (17,1);
   \draw (17,1) to [short] node [ground] {} (17,1);
   \draw (17,4) to [R=\large{\(R_{2}\)}] (21,4);
   \draw[black,fill=black] (21,4) circle (.7ex);
   \draw (21,1) to [american controlled voltage source] (21,4);
   \draw (22.5,2) node[above] {\Large{\textcolor{red}{$u_{3}$}}};
   \draw (21,1) to [short] node [ground] {} (21,1);
   \draw (9,-2) to [american controlled voltage source] (9,-5);
   \draw (7.5,-4) node[above] {\Large{\textcolor{red}{$w_{0}$}}};
   \draw (9,-5) to [short] node [ground] {} (9,-5);
   \draw[black,fill=black] (9,-2) circle (.7ex);
   \draw  (9,-2) to [R=\large{\(R_{0}\)}] (13,-2);
   \draw[black,fill=black] (13,-2) circle (.7ex);
   \draw (13,-2) node[above] {\large{$w_{1}$}};
   \draw (13,-2) to [C=\large{\(C_{1}\)}]  (13,-5);
   \draw (13,-5) to [short] node [ground] {} (13,-5);
   \draw (13,-2) to [R=\large{\(R_{1}\)}] (17,-2);
   \draw[black,fill=black] (17,-2) circle (.7ex);
   \draw (17,-2) node[above] {\large{$w_{2}$}};
   \draw (17,-2) to [C=\large{\(C_{2}\)}]  (17,-5);
   \draw (17,-5) to [short] node [ground] {} (17,-5);
   \draw (17,-2) to [R=\large{\(R_{2}\)}] (21,-2);
   \draw[black,fill=black] (21,-2) circle (.7ex);
   \draw (21,-2) node[above] {\large{$w_{3}$}};
   \draw (21,-2) to [C=\large{\(C_{3}\)}]  (21,-5);
   \draw (21,-5) to [short] node [ground] {} (21,-5);
   \draw (21,-2) to [R=\large{\(R_{3}\)}] (25,-2);
   \draw[black,fill=black] (25,-2) circle (.7ex);
   \draw (25,-2) node[above] {\large{$w_{4}$}};
   \draw (25,-2) to [C=\large{\(C_{4}\)}]  (25,-5);
   \draw (25,-5) to [short] node [ground] {} (25,-5);
   \draw (25,-2) to  (26,-2);
   \draw [dashed] (26,-2) to (27.25,-2);   
   \end{circuitikz}
   \caption{Classical WR algorithm with 2 circuit nodes overlap.}
    \label{WR_circuit}
\end{figure}
shows the decomposition of the circuit when two nodes are in the
overlap. We see that when we decompose the circuit, we need to add a
voltage source $u_3$ in the first sub-circuit and $w_0$ in the
second. At each iteration $k$, these voltage sources are given by
transmission conditions which transfer information between the
sub-circuits. The system of differential equations for the decomposed
sub-circuits are
\begin{equation}\label{wr}
\begin{array}{l}
\dot{\textbf{u}}^{k+1}(t) = \begin{bmatrix}
 			\ddots & \ddots  &\ddots &    \\
			       &  a_{n-2}      & b_{n-1}     &c_{n-1}  \\
			       &         & a_{n-1}     & b_{n} 
		      \end{bmatrix} \begin{bmatrix}
		      \vdots\\
		      u_{n-1}(t)\\
		      u_{n}(t)
		      \end{bmatrix}^{k+1}   +\begin{bmatrix}
                    \vdots \\
                    0 \\
                    c_n u_{n+1}(t)
                    \end{bmatrix}^{k+1}+\begin{bmatrix}
                    \vdots\\
                    f_{n-1}\\
                    f_{n}
                    \end{bmatrix},\\      
                    \\     
\dot{\textbf{w}}^{k+1}(t) = \begin{bmatrix}
			b_{1}      & c_{1}      &       &    \\
			a_{1}      & b_{2}      & c_{2}     &  \\
			       & \ddots       & \ddots     & \ddots 
		      \end{bmatrix} \begin{bmatrix}
		      w_{1}(t)\\
		      w_{2}(t)\\
		      \vdots
		      \end{bmatrix}^{k+1}   +\begin{bmatrix}
                    a_{0}w_{0}(t) \\
                    0\\
                    \vdots 
                    \end{bmatrix}^{k+1}\begin{bmatrix}
                    f_{1}\\
                    f_{2}\\
                    \vdots
                    \end{bmatrix},                 
\end{array}
\end{equation}
where the unknowns $u_{n+1}^{k+1}(t)$ and $w_{0}^{k+1}(t)$ are
determined by the transmission conditions
\begin{equation}\label{tc}
  u_{n+1}^{k+1}(t)=w_{n+1}^{k}(t) ,\qquad w_{0}^{k+1}(t)=u_{0}^{k}(t).
\end{equation}	
 We see that at each iteration, these transmission conditions transfer
 voltages at the interface. Comparing this with Schwarz WR applied
 to PDEs in \cite{heat_GS,heat_GL}, these conditions can be
  interpreted as Dirichlet boundary conditions.

   To start the algorithm, we specify an initial guess for the
  solution $w^{0}_{n+1}(t)$ and $u_0^{0}(t)$, and then solve the
  sub-systems \eqref{wr} for all time $t\in(0,T]$. Note that
    the two sub-systems can be solved in parallel, since in the
    transmission conditions \eqref{tc} the two sub-systems use both
    data from the previous iteration, like in a block Jacobi  method \cite{wr_jacobi}
    from linear algebra. One could also do the solves sequentially,
    and use the newest value available in the second transmission
    condition, $w_{0}^{k+1}(t)=u_{0}^{k+1}(t)$, which would be more like
    a block Gauss Seidel iteration from linear algebra, but the
    convergence analysis is similar, so we will focus on the parallel
    version here.
    
    \subsection{Convergence Analysis of the classical WR Algorithm}

The presence of different resistors $R_{i}$ and capacitors $C_{i}$
  makes the convergence analysis difficult, so to simplify, 
  we assume that all resistors and capacitors have the same 
  value, i.e. $R_{i}:=R$ and $C_{i}:=C$ for all $i\in\mathbb{Z}$
   which leads to $a_{i}=a$, $b_{i}=-2a$ and $c_{i}=a$ for all $i\in  \mathbb{Z}$.

We study the convergence of the WR algorithm \eqref{wr}
  with transmission conditions  \eqref{tc} using
  the Laplace transform:

\begin{defn}
  If $f(t)$ is a real or complex valued function of the non negative
  real variable $t$, then the Laplace transform is defined by the
  integral
  \[
    \mathcal{L}(s)=\hat{f}(s):=\int_{0}^{\infty} e^{-st}f(t)dt, \qquad s\in \mathbb{C}.
  \]
\end{defn}
 Since the system \eqref{wr} is linear, the error equations correspond
 to the homogeneous problem, $\textbf{f}=0$, with zero initial
 conditions, $\textbf{u}^{k+1}(0)=\textbf{w}^{k+1}(0)=0$. The Laplace
 transform allows us to show that
 $\hat{u}^{k+1}_j(s)=\rho_{n,cla}(s)\hat{u}^{k-1}_j(s)$ and
 $\hat{w}^{k+1}_j(s)=\rho_{n,cla}(s)\hat{w}^{k-1}_j(s)$, where
 $\hat{u}_j^k$ and $\hat{w}_j^k$ are the Laplace transforms
 of $u_j^k$ and $w_j^k$, and $\rho_{n,cla}(s)$ is the
 convergence factor. We can then use the Parseval equality for
  $s=\sigma+i\omega$, $\sigma\ge 0$, to obtain a convergence estimate
  in the weighted $L^2$ norm $\|x(t)\|_{\sigma}:=\|e^{-\sigma t}
  x(t)\|_{L^{2}}$,
\begin{equation}
\|u_{j}^{2k}(t)\|_{\sigma}\le \Big(\sup_{\omega\in \mathbb{R}} \rho_{n,cla}(s)\Big)^{k}\|u_{j}^{0}(t)\|_{\sigma}, \ \|w_{j}^{2k}(t)\|_{\sigma}\le \Big(\sup_{\omega\in \mathbb{R}} \rho_{n,cla}(s)\Big)^{k}\|w_{j}^{0}(t)\|_{\sigma}.
\end{equation}
Convergence in Laplace space with a convergence factor
  $\rho_{n,cla}(s)$ less than one implies convergence in the time
 domain in the weighted norm $\|\cdot\|_{\sigma}$, and for
 $\sigma=0$, we obtain convergence in $L^{2}$.

 We now find a closed form for the convergence factor
 $\rho_{n,cla}$. The Laplace transform for $s\in \mathbb{C}$ of the WR
 algorithm \eqref{wr} is given by
\begin{equation}\label{wrlap}
\begin{array}{l}
s\hat{\textbf{u}}^{k+1} = \begin{bmatrix}
 			\ddots & \ddots  &\ddots &    \\
			       &  a      & b     & a  \\
			       &         & a     & b 
		      \end{bmatrix} \begin{bmatrix}
		      \vdots\\
		      \hat{u}_{n-1}\\
		      \hat{u}_{n}
		      \end{bmatrix}^{k+1}   +\begin{bmatrix}
                    \vdots \\
                    0 \\
                    a\hat{w}_{n+1}^{k}
                    \end{bmatrix}, \\     
                     
 s\hat{\textbf{w}}^{k+1} = \begin{bmatrix}
			b      & a      &       &    \\
			a      & b      & a     &  \\
			       & \ddots       & \ddots     & \ddots 
		      \end{bmatrix} \begin{bmatrix}
		      \hat{w}_{1}\\
		      \hat{w}_{2}\\
		      \vdots
		      \end{bmatrix}^{k+1}   +\begin{bmatrix}
                    a\hat{u}_{0}^{k} \\
                    0\\
                    \vdots 
                    \end{bmatrix}, 
\end{array}
\end{equation}
 where we have already included the transmission conditions \eqref{tc};
 this shows the dependence of $\hat{\textbf{u}}$ on $\hat{\textbf{w}}$
 and vice-versa. Solving the sub-systems \eqref{wrlap}
 requires solving a recurrence relation of the form $a
 y_{j-1}+(b-s)y_{j}+ay_{j+1}=0$, where
 $y_{j}=\hat{u}_{j}^{k+1}$, $\hat{w}_{j}^{k+1}$ for $j\in
 \mathbb{Z}$. In order to find the convergence factor of the WR
 algorithm, we need the following lemma.
\begin{lemma}\label{lem1}
  Let $a>0$, $b<0$, $i=\sqrt{-1}$, and $s:=\sigma+i \omega$, with
  $\sigma \ge 0$. For $-b\geq 2a $, the roots
  $\lambda_{1,2}:=\frac{s-b\pm \sqrt{(b-s)^2-4a^2}}{2a}$ of the
  characteristic equation $ay_{j-1}+(b-s)y_{j}+ay_{j+1}=0$ of
  the sub-systems in \eqref{wrlap} satisfy $|\lambda_{2}| \leq 1
  \leq |\lambda_{1}|$.
\end{lemma}
\begin{proof}
The proof of this lemma can be found in \cite{RC_GK24}.
\end{proof}

\begin{theorem}\label{thwr}
 The convergence factor $\rho_{n,cla}$ of the classical WR algorithm
 \eqref{wrlap} with $n$ nodes overlap for an RC circuit of infinite
 length is given by
  \begin{equation}\label{cfwr}
    \rho_{n,cla}(s)= \left(\frac{1}{\lambda_{1}^{2}}\right)^{n+1}.
  \end{equation}
\end{theorem}
\begin{proof}
The detailed proof of this theorem can be found in
  \cite{RC_GK24}, but for completeness, we present a brief sketch of the
  proof. Solving the recurrence relation
  $ay_{j-1}+(b-s)y_{j}+ay_{j+1}=0$ for
  $y_{j}=\hat{u}_{j}^{k+1}$ gives $\hat{u}^{k+1}_{j}=A^{k+1}\lambda_{1}^j$
  for $j=(\ldots,n-2,n-1,n)$, and for $y_{j}=\hat{w}_{j}^{k+1}$
  we get $\hat{w}^{k+1}_{j}=D^{k+1}\lambda_{2}^j$ for $j\in\mathbb{N}$,
  with $\lambda_1$ and $\lambda_2$ defined in Lemma \ref{lem1}. The
  transmission conditions \eqref{tc} determine the constants $A^{k+1}$
  and $B^{k+1}$ and we find
\begin{equation}\label{eq1}
\begin{array}{l}
 \hat{u}^{k+1}_{j}=\left(\frac{-a}{a\lambda_{1}^{-1}+(b-s)}\right)\left(\frac{\hat{w}^{k}_{n+1}}{\lambda_{1}^{n}}\right)\lambda_{1}^{j} \qquad   \mbox{for $j= \ldots, n-2,n-1,n$},\\
 \hat{w}_{j}^{k+1} = \left(\frac{-a \hat{u}^{k}_{0}}{(b-s)+a\lambda_{2}}\right) \lambda_{2}^{j-1} \qquad  \mbox{ for $j\in\mathbb{N}$}. 
 \end{array}
\end{equation}
 The above expression together with Vieta's formulas
  $\lambda_1+\lambda_2 = (s-b)/a $ and $\lambda_1 \lambda_2= 1$ gives
  $\hat{u}_{j}^{k+1} = \rho_{n,cla}(s)\hat{u}_{j}^{k-1}$ for
  $j=(\ldots,n-1,n)$ and $\hat{w}_{j}^{k+1} =
  \rho_{n,cla}(s)\hat{w}_{j}^{k-1}$ for $j\in\mathbb{N}$ where
  $\rho_{n,cla}(s)$ is given in \eqref{cfwr}.
\end{proof}

From Lemma \ref{lem1}, we see that if $|b|=2a$ and $s=0$,
  the convergence factor $|\rho_{n,cla}(0)|=1$, while for
  $b=-(2+\epsilon)a$ with $\epsilon>0$, we have $|\rho_{n,cla}(s)|<1$.
  It is therefore interesting to study the case $\epsilon>0$ and
  then analyze the limit as $\epsilon\rightarrow0$. Infact, in
  this limit, the rate of convergence deteriorates at $\omega=0$, as
  we show on the left in Fig. \ref{fig_3_wr}.
\begin{figure}
  \includegraphics[width=0.49\textwidth,clip]{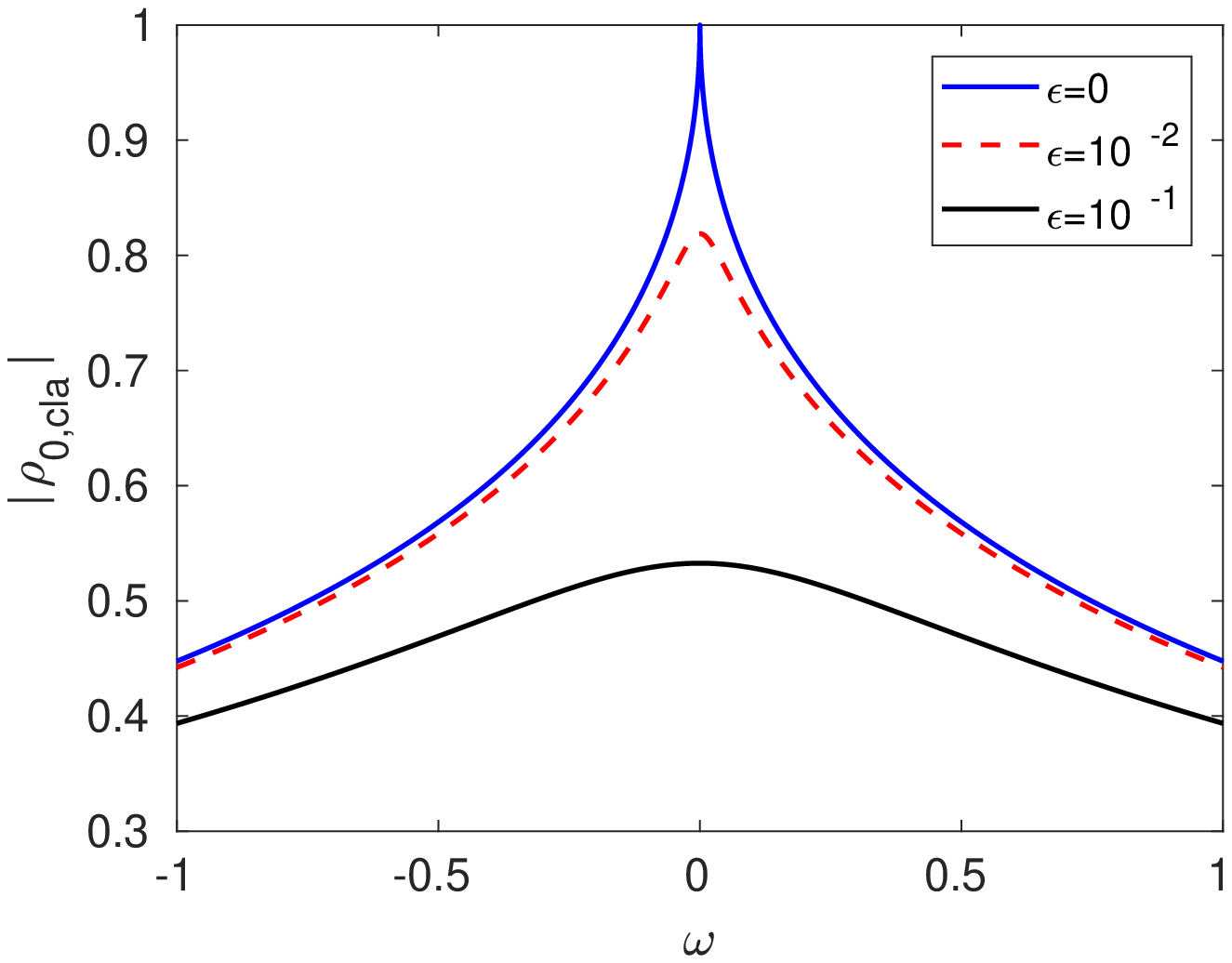}
  \includegraphics[width=0.49\textwidth,clip]{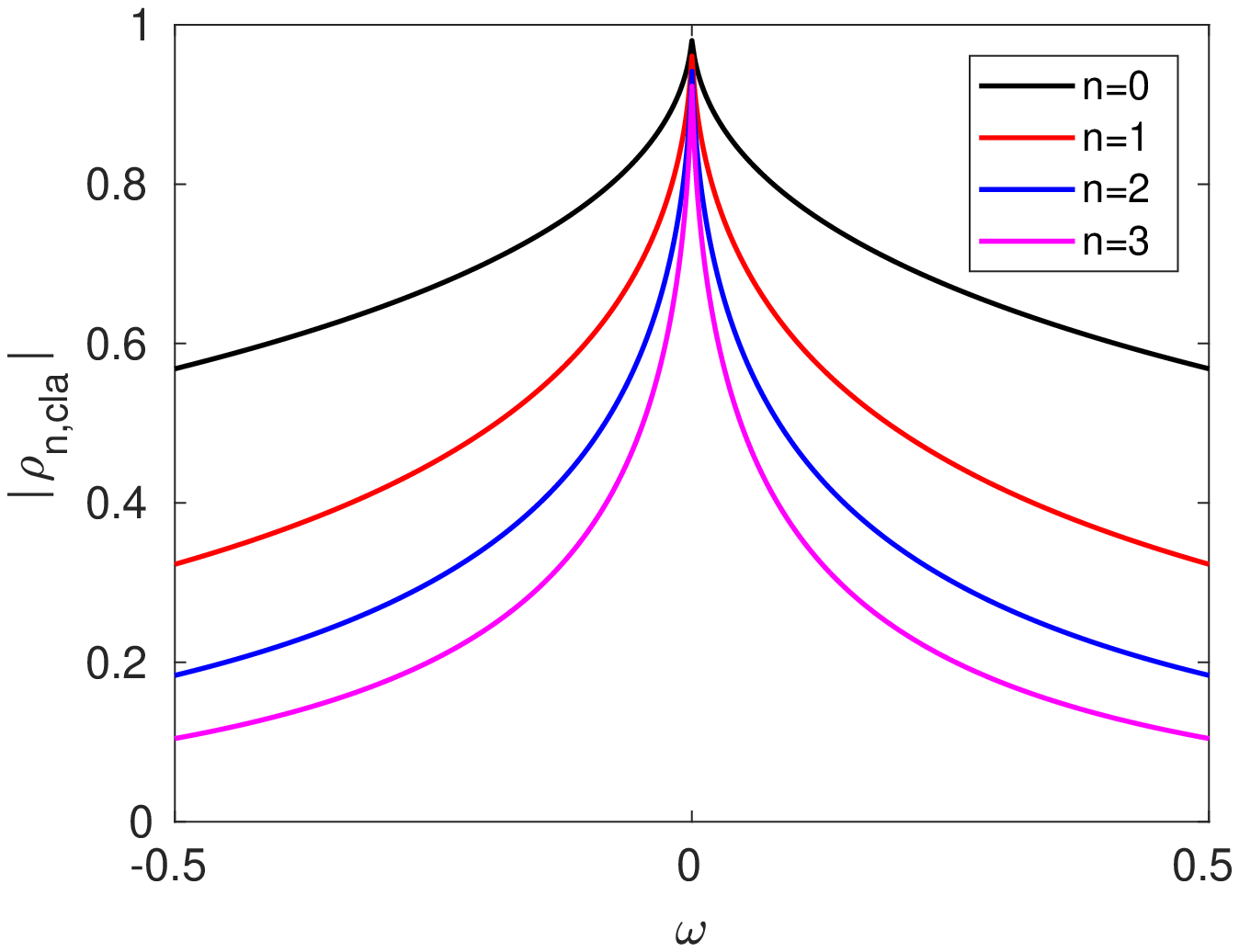}
  \caption{Convergence factor for $s=i\omega$ for $n=0$ and different values
    of $\epsilon$ (left) and for different overlaps with
      $\epsilon=1e-4$ (right).}
  \label{fig_3_wr}
\end{figure}
We also show on the right in Fig. \ref{fig_3_wr} that increasing
 the overlap increases the rate of convergence, but the effect is very
 small for $\omega$ close to zero when $s=i\omega$.

For circuits, the introduction of $\epsilon$ leads to the addition of
 a resistor $\tilde{R}=R/\epsilon$ at each node of the circuit (see
 Fig. \ref{eps_WR_circuit}).
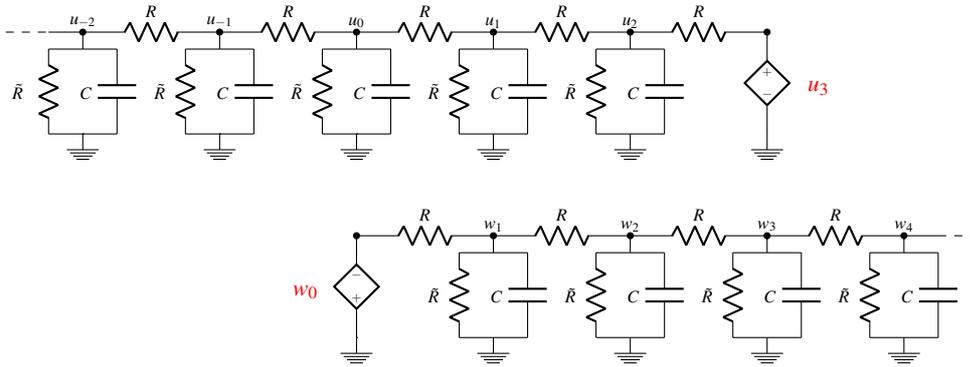
\begin{figure}[t]
  \centering
  \begin{circuitikz}[scale=0.45, every node/.style={scale=0.6}]
   \draw [dashed] (-1.25,4) to (0,4);
   \draw (0,4) to (1,4);
   \draw [black,fill=black] (1,4) circle (.7ex);
   \draw (1,4) node[above] {\large{$u_{-2}$}};
   \draw (1,4) to (1,3.5);
   \draw (1,3.5) to (2,3.5);
   \draw (2,3.5) to [C]  (2,1);
   \draw (1.1,1.9) node[above] {\large{$C$}};
   \draw (2,1) to (1,1);
   \draw (1,3.5) to (0,3.5);
   \draw (0,3.5) to [R]  (0,1);
   \draw (-0.9,1.9) node[above] {\large{$\tilde{R}$}};
   \draw (0,1) to (1,1);
   \draw (1,1) to [short] node [ground] {} (1,1);
   \draw (1,4) to [R=\large{\(R\)}] (5,4);
   \draw [black,fill=black] (5,4) circle (.7ex);
   \draw (5,4) node[above] {\large{$u_{-1}$}};
   \draw (5,4) to (5,3.5);
   \draw (5,3.5) to (6,3.5);
   \draw (6,3.5) to [C]  (6,1);
   \draw (5.1,1.9) node[above] {\large{$C$}};
   \draw (6,1) to (5,1);
   \draw (5,3.5) to (4,3.5);
   \draw (4,3.5) to [R]  (4,1);
   \draw (3.25,1.9) node[above] {\large{$\tilde{R}$}};
   \draw (4,1) to (5,1);
   \draw (5,1) to [short] node [ground] {} (5,1);
   \draw (5,4) to [R=\large{\(R\)}] (9,4);
   \draw [black,fill=black] (9,4) circle (.7ex);
   \draw (9,4) node[above] {\large{$u_{0}$}};
   \draw (9,4) to (9,3.5);
   \draw (9,3.5) to (10,3.5);
   \draw (10,3.5) to [C]  (10,1);
   \draw (9.1,1.9) node[above] {\large{$C$}};
   \draw (10,1) to (9,1);
   \draw (9,3.5) to (8,3.5);
   \draw (8,3.5) to [R]  (8,1);
   \draw (7.25,1.9) node[above] {\large{$\tilde{R}$}};
   \draw (9,1) to (8,1);
   \draw (9,1) to [short] node [ground] {} (9,1);
   \draw (9,4) to [R=\large{\(R\)}] (13,4);
   \draw [black,fill=black] (13,4) circle (.7ex);   
   \draw (13,4) node[above] {\large{$u_{1}$}};
   \draw (13,4) to (13,3.5);
   \draw (13,3.5) to (14,3.5);
   \draw (14,3.5) to [C]  (14,1);
   \draw (13.1,1.9) node[above] {\large{$C$}};
   \draw (14,1) to (13,1);
   \draw (13,3.5) to (12,3.5);
   \draw (12,3.5) to [R]  (12,1);
   \draw (11.25,1.9) node[above] {\large{$\tilde{R}$}};
   \draw (12,1) to (13,1);
   \draw (13,1) to [short] node [ground] {} (13,1);
   \draw (13,4) to [R=\large{\(R\)}] (17,4);
   \draw [black,fill=black] (17,4) circle (.7ex);
   \draw (17,4) node[above] {\large{$u_{2}$}};
   \draw (17,4) to (17,3.5);
   \draw (17,3.5) to (18,3.5);
   \draw (18,3.5) to [C]  (18,1);
   \draw (17.1,1.9) node[above] {\large{$C$}};
   \draw (18,1) to (17,1);
   \draw (17,3.5) to (16,3.5);
   \draw (16,3.5) to [R]  (16,1);
   \draw (15.25,1.9) node[above] {\large{$\tilde{R}$}};
   \draw (16,1) to (17,1);
   \draw (17,1) to [short] node [ground] {} (17,1);
   \draw (17,4) to [R=\large{\(R\)}] (21,4);
   \draw[black,fill=black] (21,4) circle (.7ex);
   \draw (21,1) to [american controlled voltage source] (21,4);
   \draw (22.5,2) node[above] {\Large{\textcolor{red}{$u_{3}$}}};
   \draw (21,1) to [short] node [ground] {} (21,1);
   \draw[black,fill=black] (9,-2) circle (.7ex);
   \draw (9,-2) to [american controlled voltage source] (9,-5);
   \draw (7.5,-4) node[above] {\Large{\textcolor{red}{$w_{0}$}}};
   \draw (9,-5) to [short] node [ground] {} (9,-5);
   \draw  (9,-2) to [R=\large{\(R\)}] (13,-2);
   \draw[black,fill=black] (13,-2) circle (.7ex);
   \draw (13,-2) node[above] {\large{$w_{1}$}};
   \draw  (13,-2) to (13,-2.5);
   \draw  (13,-2.5) to (14,-2.5);
   \draw (14,-2.5) to [C]  (14,-5);
   \draw (13.1,-4.1) node[above] {\large{$C$}};
   \draw  (14,-5) to (13,-5);
   \draw  (13,-5) to (12,-5);
   \draw (12,-2.5) to [R]  (12,-5);
   \draw (11.25,-4.1) node[above] {\large{$\tilde{R}$}};
   \draw  (12,-2.5) to (13,-2.5);
   \draw (13,-5) to [short] node [ground] {} (13,-5);
   \draw (13,-2) to [R=\large{\(R\)}] (17,-2);
   \draw[black,fill=black] (17,-2) circle (.7ex);
   \draw (17,-2) node[above] {\large{$w_{2}$}};
   \draw  (17,-2) to (17,-2.5);
   \draw  (17,-2.5) to (18,-2.5);
   \draw (18,-2.5) to [C]  (18,-5);
   \draw (17.1,-4.1) node[above] {\large{$C$}};
   \draw  (18,-5) to (17,-5);
   \draw  (17,-5) to (16,-5);
   \draw (16,-2.5) to [R]  (16,-5);
   \draw (15.25,-4.1) node[above] {\large{$\tilde{R}$}};
   \draw  (16,-2.5) to (17,-2.5);
   \draw (17,-5) to [short] node [ground] {} (17,-5);
   \draw (17,-2) to [R=\large{\(R\)}] (21,-2);
   \draw[black,fill=black] (21,-2) circle (.7ex);
   \draw (21,-2) node[above] {\large{$w_{3}$}};
    \draw  (21,-2) to (21,-2.5);
   \draw  (21,-2.5) to (22,-2.5);
   \draw (22,-2.5) to [C]  (22,-5);
   \draw (21.1,-4.1) node[above] {\large{$C$}};
   \draw  (22,-5) to (21,-5);
   \draw  (21,-5) to (20,-5);
   \draw (20,-2.5) to [R]  (20,-5);
   \draw (19.25,-4.1) node[above] {\large{$\tilde{R}$}};
   \draw  (20,-2.5) to (21,-2.5);
   \draw (21,-5) to [short] node [ground] {} (21,-5);
   \draw (21,-2) to [R=\large{\(R\)}] (25,-2);
   \draw[black,fill=black] (25,-2) circle (.7ex);
   \draw (25,-2) node[above] {\large{$w_{4}$}};
   \draw (25,-2) to (25,-2.5);
   \draw (25,-2.5) to (26,-2.5);
   \draw (26,-2.5) to [C]  (26,-5);
   \draw (25.1,-4.1) node[above] {\large{$C$}};
   \draw  (26,-5) to (25,-5);
   \draw  (25,-5) to (24,-5);
   \draw (24,-2.5) to [R]  (24,-5);
   \draw (23.25,-4.1) node[above] {\large{$\tilde{R}$}};
   \draw  (24,-2.5) to (25,-2.5);
   \draw (25,-5) to [short] node [ground] {} (25,-5);
   \draw (25,-2) to  (26,-2);
   \draw [dashed] (26,-2) to (27.25,-2);   
   \end{circuitikz}
   \caption{WR algorithm with $b=-(2+\epsilon)a$ and
     $\tilde{R}=R/\epsilon$}
    \label{eps_WR_circuit}
\end{figure}
  The inclusion of the resistors $\tilde{R}$ presents a practical case where there is a small 
  leakage of current in the dielectric current. The limit $\epsilon\rightarrow 0$ leads to the limit
 $\tilde{R}\rightarrow \infty$, which means that no current passes
 through this resistor. Thus considering $b=-(2+\epsilon)a$ and taking
 the limit $\epsilon\rightarrow 0$ states that the circuit in
 Fig. \ref{eps_WR_circuit} is a good approximation to circuit in
 Fig. \ref{WR_circuit}. Adding $\epsilon>0$ also corresponds to convergence  in a weighted norm, with choosing $\sigma=\epsilon>0$ in $s$. This can be seen be substituting $w(t)=u(t) e^{\epsilon t}$.

 \section{OWR Algorithm}\label{Sec:OWR}

The main drawback of the classical WR algorithm is that it is very
 slow especially when large time windows are used. Large time windows
 in real space correspond to small frequencies $\omega$ in
  $s=\sigma+i\omega$ in Laplace space. When $\epsilon\rightarrow
 0$ with $b=-(2+\epsilon)a$, we observe that
 $|\rho_{n,cla}(0)|\rightarrow 1$, which means the convergence rate
 slows down for small $\omega$, and increasing the overlap
 does mostly improve the convergence of higher frequencies
  $\omega$, as one can see in Fig. \ref{fig_3_wr} on the right. The
 Dirichlet transmission conditions \eqref{tc} which exchange just
 voltages at the interfaces are the main reason for this slow
 convergence. We thus search for better transmission conditions to
 exchange information between the sub-circuits. One way is to use
 optimized transmission conditions, which are defined for this RC
 circuit in \cite{RC_GK24} as
\begin{equation}\label{otc}
\begin{array}{l}
(u_{n+1}^{k+1}-u_{n}^{k+1})+\alpha u_{n+1}^{k+1}=(w_{n+1}^{k}-w_{n}^{k})+\alpha w_{n+1}^{k}, \\
(w_{1}^{k+1}-w_{0}^{k+1})+\beta w_{0}^{k+1}=(u_{1}^{k}-u_{0}^{k})+\beta u_{0}^{k},
\end{array}
\end{equation}
 where $\alpha$, $\beta \in \mathbb{R}$ and $k$ is the iteration
 index. These transmission conditions exchange both voltages and
 currents at the interface, which can be seen by dividing the
 first equation by $\alpha$ and the second by $\beta$. The term
 $\frac{u^{k+1}_{n+1}-u^{k+1}_{n}}{\alpha}$ can be viewed as a
 current and $u^{k+1}_{n+1}$ is a voltage. These transmission
 conditions are called optimized transmission conditions since we
  need to find the best (optimized) values for $\alpha$ and
  $\beta$ such that the convergence factor is as small as possible.

For circuits, the introduction of the new transmission
 conditions \eqref{otc} means two voltage sources need to be added at
 the interface of each sub-circuit. From Fig. \ref{OWR_circuit},
\begin{figure}
  \centering
  \begin{circuitikz}[scale=0.45, every node/.style={scale=0.7}]
   \draw [dashed] (-1.25,4) to (0,4);
   \draw (0,4) to (1,4);
   \draw [black,fill=black] (1,4) circle (.7ex);
   \draw (1,4) node[above] {\large{$u_{-2}$}};
   \draw (1,1) to [C=\large{\(C\)}]  (1,4);
   \draw (1,1) to [short] node [ground] {} (1,1);
   \draw (1,4) to [R=\large{\(R\)}] (5,4);
   \draw [black,fill=black] (5,4) circle (.7ex);
   \draw (5,4) node[above] {\large{$u_{-1}$}};
   \draw (5,1) to [C=\large{\(C\)}]  (5,4);
   \draw (5,1) to [short] node [ground] {} (5,1);
   \draw (5,4) to [R=\large{\(R\)}] (9,4);
   \draw [black,fill=black] (9,4) circle (.7ex);
   \draw (9,4) node[above] {\large{$u_{0}$}};
   \draw (9,1) to [C=\large{\(C\)}]  (9,4);
   \draw (9,1) to [short] node [ground] {} (9,1);
   \draw (9,4) to [R=\large{\(R\)}] (13,4);
   \draw [black,fill=black] (13,4) circle (.7ex);
   \draw (13,4) node[above] {\large{$u_{1}$}};
   \draw (13,1) to [C=\large{\(C\)}]  (13,4);
   \draw (13,1) to [short] node [ground] {} (13,1);
   \draw (17,4) to [R=\large{\(R\)}] (13,4);
   \draw (17,1) to [american controlled voltage source] (17,4);
   \draw (18.25,2) node[above] {\Large{\textcolor{red}{$w_{2}$}}};
   \draw (17,1) to [short] node [ground] {} (17,1);
   \draw (13,4) to (15,5);
   \draw (15,5) to [R=\large{\(R(1+\alpha)\)}] (19,5);
   \draw (19,5) to (20,5);   
   \draw (20,5) to [american controlled voltage source] (20,1);
   \draw (21.25,2.5) node[above] {\Large{\textcolor{red}{$w_{1}$}}};
   \draw (20,1) to [short] node [ground] {} (20,1);    
   \draw (13,-3) to [R=\large{\(R\)}] (9,-3);
   \draw (9,-3) to [american controlled voltage source] (9,-6);
   \draw (7.75,-5) node[above] {\Large{\textcolor{red}{$u_{0}$}}};
   \draw (9,-6) to [short] node [ground] {} (9,-6);
   \draw (13,-3) to (11,-2);
   \draw (7,-2) to [R=\large{\(R(1-\beta)\)}] (11,-2);
   \draw (7,-2) to (6,-2);   
   \draw (6,-6) to [american controlled voltage source] (6,-2);
   \draw (4.75,-4.5) node[above] {\Large{\textcolor{red}{$u_{1}$}}};
   \draw (6,-6) to [short] node [ground] {} (6,-6);       
   \draw[black,fill=black] (13,-3) circle (.7ex);
   \draw (13,-3) node[above] {\large{$w_{1}$}};
   \draw (13,-3) to [C=\large{\(C\)}]  (13,-6);
   \draw (13,-6) to [short] node [ground] {} (13,-6);
   \draw (13,-3) to [R=\large{\(R\)}] (17,-3);
   \draw[black,fill=black] (17,-3) circle (.7ex);
   \draw (17,-3) node[above] {\large{$w_{2}$}};
   \draw (17,-3) to [C=\large{\(C\)}]  (17,-6);
   \draw (17,-6) to [short] node [ground] {} (17,-6);
   \draw (17,-3) to [R=\large{\(R\)}] (21,-3);
   \draw[black,fill=black] (21,-3) circle (.7ex);
   \draw (21,-3) node[above] {\large{$w_{3}$}};
   \draw (21,-3) to [C=\large{\(C\)}]  (21,-6);
   \draw (21,-6) to [short] node [ground] {} (21,-6);
   \draw (21,-3) to [R=\large{\(R\)}] (25,-3);
   \draw[black,fill=black] (25,-3) circle (.7ex);
   \draw (25,-3) node[above] {\large{$w_{4}$}};
   \draw (25,-3) to [C=\large{\(C\)}]  (25,-6);
   \draw (25,-6) to [short] node [ground] {} (25,-6);
   \draw (25,-3) to  (26,-3);
   \draw [dashed] (26,-3) to (27.25,-3);   
   \end{circuitikz}
   \caption{OWR algorithm with 1 circuit node overlap.}
    \label{OWR_circuit}
\end{figure}
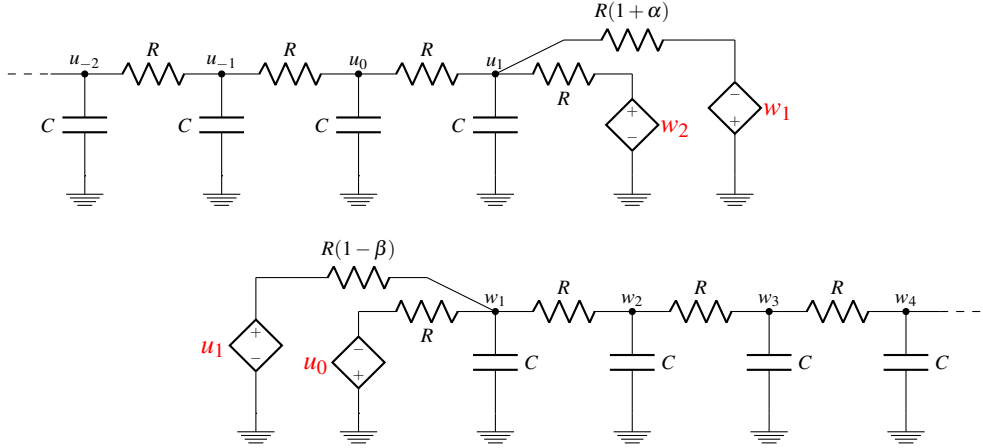
 we see that the resistors $R_{\alpha}:=R(1+\alpha)$ and
 $R_\beta:=R(1-\beta)$, which depend on the parameters $\alpha$
 and $\beta$, are added.  Similar to the WR algorithm, at each
 iteration the voltage sources $w_3$, $w_2$, $u_0$, $u_1$ are
 transferred between the sub-circuits.

 We can also interpret the new transmission conditions \eqref{otc}
  as Robin transmission conditions for the discretized heat equation
  \eqref{HeatEquationDiscrete}: if we divide the first equation of
 \eqref{otc} by $\bar{\alpha}:=\frac{\alpha}{p}$ for some $p>0$ and the
 second equation by $\bar{\beta}:=\frac{-\beta}{p}$, we obtain
\begin{eqnarray*}
\frac{u_{n+1}^{k+1}-u_{n}^{k+1}}{\bar{\alpha}}+p u_{n+1}^{k+1}&=&\frac{w_{n+1}^{k}-w_{n}^{k}}{\bar{\alpha}}+p w_{n+1}^{k}, \\
\frac{w_{1}^{k+1}-w_{0}^{k+1}}{\bar{\beta}}-p w_{0}^{k+1}&=&\frac{u_{1}^{k}-u_{0}^{k}}{\bar{\beta}}-p u_{0}^{k}.
\end{eqnarray*}
  If we consider now $\bar{\alpha}\approx \Delta x$ and
  $\bar{\beta}\approx \Delta x$, then the fractions inside the above
  equations represent discretization for the derivative $\frac{\partial
  u}{\partial x}$ and $\frac{\partial w}{\partial x}$, and
  we thus obtain in the limit Robin transmission conditions,
\begin{eqnarray*}
\left(\frac{\partial }{\partial x}+p\right)u_{n+1}^{k+1}&=&\left(\frac{\partial }{\partial x}+p\right) w_{n+1}^{k}, \\
\left(\frac{\partial }{\partial x}-p\right) w_{0}^{k+1}&=&\left(\frac{\partial }{\partial x}-p\right) u_{0}^{k}.
\end{eqnarray*} 

\subsection{Convergence of OWR}

As for the classical WR algorithm, we analyze the convergence of
  the OWR algorithm in Laplace space. We first rearrange the
  optimized transmission conditions as
\begin{equation}\label{eq2}
\begin{array}{l}
 u_{n+1}^{k+1}=\frac{u_{n}^{k+1}}{1+\alpha}+w_{n+1}^{k}-\frac{w_{n}^{k}}{1+\alpha}, \\ 
w_{0}^{k+1}=-\frac{w_{1}^{k+1}}{\beta-1}+u_{0}^{k} + \frac{u_{1}^{k}}{\beta-1},
\end{array}
\end{equation}
  where all the voltages $u_{i}$ and $w_{i}$ depend on time $t$. We
  substitute these rearranged transmission conditions into \eqref{wr}
  and take the Laplace transform to arrive at
\begin{equation}\label{owr}
\begin{array}{l}
s \hat{\textbf{u}}^{k+1} = \begin{bmatrix}
 			\ddots & \ddots  &\quad \ddots &    \\
			       &  a & b & a  \\
			       &         & a & b  +\frac{a}{\alpha+1}
		      \end{bmatrix} \begin{bmatrix}
		      \vdots\\
		      \hat{u}_{n-1}\\
		      \hat{u}_{n}
		      \end{bmatrix}^{k+1}   +\begin{bmatrix}
                    \vdots \\
                    0 \\
                    a\hat{w}_{n+1}^{k}-\frac{a}{\alpha +1}\hat{w}_{n}^{k}
                    \end{bmatrix},\\         
                     
s \hat{\textbf{w}}^{k+1} = \begin{bmatrix}
			b-\frac{a}{\beta-1}      & a      &       &    \\
			a      & b      & a     &  \\
			       & \ddots       & \ddots     & \ddots 
		      \end{bmatrix} \begin{bmatrix}
		      \hat{w}_{1}\\
		      \hat{w}_{2}\\
		      \vdots
		      \end{bmatrix}^{k+1}   +\begin{bmatrix}
                    a\hat{u}_{0}^{k}+\frac{a}{\beta-1}\hat{u}_{1}^{k} \\
                    0\\
                    \vdots 
                    \end{bmatrix}.                  
\end{array}
\end{equation}
 We need to give initial guesses $\hat{w}^{0}_{n+1}$, $\hat{w}^{0}_{n}$,
  $\hat{u}^{0}_{0}$ and $\hat{u}^{0}_{1}$ to start this algorithm.

\begin{theorem}\label{thowr} 
The convergence factor $\rho_{n}(s,\alpha,\beta)$ of the OWR algorithm
\eqref{owr} for an RC circuit of infinite length is given by
\begin{equation}\label{cfowr}
\rho_{n}(s,\alpha,\beta)=\left(\frac{\alpha+1 -\lambda_{1}}{\lambda_{1}(1+\alpha)-1}\right)\left(\frac{\lambda_{1}+\beta-1}{1+(\beta-1)\lambda_{1}}\right)\left(\frac{1}{\lambda_{1}^2}\right)^n.
\end{equation}
\end{theorem}
\begin{proof}
To find the convergence factor, we proceed as in the proof of
 Theorem \ref{thwr} to arrive at
  $\hat{u}^{k+1}_{j}=A^{k+1}\lambda_{1}^{j}$ for $j=(\dots,n-1,n)$, and
  $\hat{w}^{k+1}_{j}=D^{k+1}_{j}\lambda_{2}^{j}$ for $j\in\mathbb{N}$. To
  determine the constants $A^{k+1}$ and $D^{k+1}$, we use the
  optimized transmission conditions \eqref{eq2} and get
\begin{equation}\label{eq3}
 A^{k+1}=B^{k}\left(\frac{\lambda_{2}(1+\alpha)-1}{\lambda_{1}(1+\alpha)-1}\right)\left(\frac{\lambda_{2}}{\lambda_{1}}\right)^{n},\quad
 B^{k+1}=A^{k}\left(\frac{\lambda_{1}+\beta-1}{\lambda_{2}+\beta-1}\right).
\end{equation}
 Using $\lambda_{1}\lambda_{2}=1$ and equations \eqref{eq3}, we arrive at 
\begin{eqnarray*}
\hat{u}_{j}^{k+1 }&=& A^{k+1} \lambda_{1}^{j} \qquad  \quad j\in (\dots,-2,-1,0,1,\dots,n) \\
 &=& \left(\frac{\lambda_{2}(1+\alpha)-1}{\lambda_{1}(1+\alpha)-1}\right)\left(\frac{\lambda_{1}+\beta-1}{\lambda_{2}+\beta-1}\right)\left(\frac{\lambda_{2}}{\lambda_{1}}\right)^{n} A^{k-1}\lambda_{1}^{j}\\ 
&=& \left(\frac{\alpha+1 -\lambda_{1}}{\lambda_{1}(1+\alpha)-1}\right)\left(\frac{\lambda_{1}+\beta-1}{1+(\beta-1)\lambda_{1}}\right)\left(\frac{1}{\lambda_{1}^2}\right)^{n}\hat{u}_{j}^{k-1} \\
&=& \rho_{n}(s,\alpha,\beta) \hat{u}_{j}^{k-1},
 \end{eqnarray*}
 where the convergence factor $\rho_{n}(s,\alpha,\beta)$ is given by
 \eqref{cfowr}.  Similarly, we can also show that
 $\hat{w}_{j}^{k+1}=\rho_{n}(s,\alpha,\beta) \hat{w}_{j}^{k-1}$, for
 $j\in\mathbb{N}$, and this completes the proof.
 \end{proof}

\begin{lemma}
If $\alpha>0$, $\beta<0$, and $\epsilon>0$, then the modulus
of the convergence factor $\rho_{n}(s,\alpha,\beta)$ of the OWR
algorithm is less than 1, that is, $|\rho_{n}(s,\alpha,\beta)|<1$ for all $n\ge
0$.
\end{lemma}
\begin{proof}
 Since $\lambda_{1}\in \mathbb{C}$, we assume $\lambda_1=x+iy$.
  Lemma \ref{lem1} states that for $\epsilon>0$, where 
  $b=-(2+\epsilon)a$, we have $|\lambda_{1}|>1$ and hence
   $x^2+y^2>1$. Further for $\alpha>0$, $(\alpha+1)^2-1>0$ and hence
\begin{eqnarray*}
&&(\alpha+1)^2-1 <  [(\alpha+1)^2-1](x^2+y^2)\\
&\iff &(\alpha+1)^2+x^2+y^2 < (\alpha+1)^2 x^2+(\alpha+1)^2 y^2+1\\
&\iff &(\alpha+1)^2+x^2+y^2-2 x (\alpha+1) <  (\alpha+1)^2 x^2+(\alpha+1)^2 y^2+1-2 x (\alpha+1)\\
&\iff & (\alpha+1-x)^2 + y^2 < ((\alpha+1) x-1)^2 + (\alpha+1)^2 y^2\\
&\iff & |\alpha+1-x-i y| < |(\alpha+1)(x+i y)-1|.
\end{eqnarray*}
  Similarly, for $\beta<0$, we can show $|\lambda_1+\beta-1|< |1+(\beta-1)       \lambda_1|$ and this completes the proof.
\end{proof}

  We observe from \eqref{cfowr} that the effect of
   overlap on the convergence factor given by $\left(\frac{1}{\lambda_{1}^2}\right)^n$ is the same for both the WR and the OWR algorithm.
    This means increasing the overlap increases the rate of
   convergence also for OWR. The convergence factor is also
    the same for all the circuit nodes irrespective of which 
  sub-circuit they belong to. Further, for fast convergence,
   we would like to have the convergence factor
    $|\rho_{n}(s,\alpha,\beta)|$ as small as possible. 
   The parameters $a$, $b$ represent circuit elements and cannot
   be changed, but we can choose $\alpha$ and $\beta$ such that
  $|\rho_{n}(s,\alpha,\beta)|$ becomes as small as possible.
 So what is the best possible choice for the parameters
  $\alpha$ and $\beta$?

\begin{theorem}\label{th_optimal}
For the circuit decomposition into two sub-circuits with only one interface, the Optimized Waveform Relaxation method converges in two iterations,
independently of initial guess and the overlap, if
\begin{equation}\label{opt_alpha}
  \alpha_{opt}:=\lambda_1 -1 \qquad \mbox{and} \qquad \beta_{opt}:=1-\lambda_1.
\end{equation}
\end{theorem}
\begin{proof}
  Setting the convergence factor $\rho_{n}(s,\alpha,\beta)=0$, we find
  $\alpha=\lambda_1-1$ and $\beta=1-\lambda_1$. Since,
  $\hat{u}^{k+1}_{j}=\rho_{n}(s,\alpha,\beta)\hat{u}^{k-1}_{j}$ and
  $\hat{w}^{k+1}_{j}=\rho_{n}(s,\alpha,\beta)\hat{w}^{k-1}_{j}$, we have
  $\hat{u}^{2}_{j}$ and $\hat{w}^{2}_{j}$ identically zero and hence
  the OWR has converged in two iterations.
\end{proof}
 One can see that this is the best choice, since the solution
  in each sub-system depends on all the source terms $f_j$ and
  during the first iteration, each sub-system uses only the
  local $f_j$ to compute the approximation. It is only for the
  second iteration that information is transferred. Therefore,
  convergence cannot be achieved in less than 2 iterations.

  Since $\lambda_1$ is a complicated function of $s\in \mathbb{C}$, its
  inverse Laplace transform leads to non local operators in time
  for $\alpha_{opt}$ and $\beta_{opt}$. These non local operators
  are expensive to use since they require convolution operations. It
  is therefore of interest to approximate $\alpha_{opt}$ and
  $\beta_{opt}$ by a polynomial in $s$. In this paper, we will focus on
  approximation of $\alpha_{opt}$ and $\beta_{opt}$ by a constant.
  
\section{Optimization}\label{Sec:Optimization}

  Mathematically, we want $|\rho_{n}(s,\alpha,\beta)|\ll 1$, which leads
  to solving the min-max problem
\begin{equation}\label{min_max}
  \min_{\alpha,\beta}\left(\max_{s}|\rho_{n}(s,\alpha,\beta)|\right).
\end{equation}
  Since $s\in\mathbb{C}$ with $s=\sigma+i \omega$, the above
  optimization problem is in four variables and hence already very
  difficult to solve. We simplify the problem further using some
  assumptions and the following lemmas.

\begin{lemma}\label{max_prin}
  For $\alpha>0$, $\beta<0$, the maximum of $|\rho_{n}(s,\alpha,\beta)|$
  lies on the imaginary axis of the complex plane.
\end{lemma}
\begin{proof}
  The detailed proof of this lemma can be found in \cite{RC_GK24}. The
  idea is to show that $\rho_{n}(s,\alpha,\beta)$ is analytic in the
  right half of the complex plane and then to use the maximum modulus
  principle for analytic functions.
\end{proof}

\begin{lemma}\label{sym}
  For $\sigma=0$, $|\rho_{n}(\omega,\alpha,\beta)|$ is symmetric in
  $\omega$.
\end{lemma}
\begin{proof}
  On the imaginary axis of the complex plane, $\sigma=0$, and hence from
  the definition of $\lambda_{1}$, we get
\begin{equation*}
  \lambda_{1}(\omega)=\frac{i\omega -b + \sqrt{(i\omega- b)^2-4a^2}}{2a} 
    =\frac{i\omega +(2+\epsilon)a + \sqrt{r + i p}}{2a},
\end{equation*}
  where $r:=\epsilon^2 a^2 - \omega^2 + 4\epsilon a^2$ and
  $p:=2(2+\epsilon)\omega a$. Further, letting $z_1+i 
  z_2:=\sqrt{r+i p}$, where $z_1$, $z_2 \in\mathbb{R}$, we get
  $\lambda_{1}(\omega)=x+iy=\frac{(2+\epsilon)a+z_1}{2a} + i
  \frac{\omega+z_2}{2a}$. One can check that if $\sqrt{r+i p}=z_1+i
  z_2$ then $\sqrt{r-i p}=z_1-i z_2$. Therefore,
  $\lambda_{1}(-\omega)=\frac{(2+\epsilon)a+z_1}{2a} - i
  \frac{\omega+z_2}{2a}$ which shows that
  $\overline{\lambda_{1}(-\omega)} =\lambda_{1}(\omega)$. Now,
\begin{eqnarray*}
|\rho_{n}(-\omega,\alpha,\beta)|&=&\left|\left(\frac{\alpha+1
    -\lambda_{1}(-\omega)}{[\lambda_{1}(-\omega)](1+\alpha)-1}\right)     \left(\frac{\lambda_{1}(-\omega)+\beta-1}{1+(\beta-1)[\lambda_{1}(-\omega)]}   \right)\left(\frac{1}{[\lambda_{1}(-\omega)]^2}\right)^n\right|\\
 &=&\left|\left(\frac{\overline{\alpha+1
    -\lambda_{1}(\omega)}}{\overline{[\lambda_{1}(\omega)](1+\alpha)-1}}\right)\right|\left|\left(\frac{\overline{\lambda_{1}(\omega)+\beta-1}}{\overline{1+(\beta-1)[\lambda_{1}(\omega)]}}\right)\right|\left|\left(\frac{1}{[\overline{\lambda_{1}(\omega)}]^2}\right)^n\right| \\
 &=&\left|\left(\frac{\alpha+1
    -\lambda_{1}(\omega)}{[\lambda_{1}(\omega)](1+\alpha)-1}\right)\right|\left|\left(\frac{\lambda_{1}(\omega)+\beta-1}{1+(\beta-1)[\lambda_{1}(\omega)]}\right)\right|\left|\left(\frac{1}{[\lambda_{1}(\omega)]^2}\right)^n\right| \\
    &=&|\rho_{n}(\omega,\alpha,\beta)|,
\end{eqnarray*}
  which concludes the proof.
\end{proof}

From Theorem \ref{th_optimal}, we observe that $\alpha_{opt}$
  and $\beta_{opt}$ are related to each other via the relation
  $\beta_{opt}=-\alpha_{opt}$, which suggests the natural
  assumption $\beta=-\alpha$. In our RC circuit of infinite length,
  this would mean that at the interface (where the circuit is split into
  two), the current flowing in both sub-circuits is equal but into
  opposite direction. This interpretation is easy to see for the
  non-overlapping case $n=0$. Recall that the terms
  $\frac{u^{k+1}_{n+1}-u^{k+1}_{n}}{\alpha}$ and
  $\frac{w^{k+1}_{1}-w^{k+1}_{0}}{\beta}$ are viewed as currents. Thus
  with $\beta=-\alpha$, their values are same but their sign is
  opposite.

Lemma \ref{max_prin} and Lemma \ref{sym} state that the min-max
  problem \eqref{min_max} needs to be solved for $s=i\omega$, $\omega\ge
 0$. However, for numerical calculations, we consider the time $t\in
  [0,T]$, and also a discretization with $\Delta t$ as the
  discretization parameter. Hence $\omega_{min}\le \omega \le
  \omega_{max}$, where we can estimate $\omega_{min}=\frac{\pi}{T}$ 
  and $\omega_{max}=\frac{\pi}{\Delta t}$. Note that 
  $\omega_{min}>0$, but to further simplify our analysis,
  we consider a wider range for $\omega$, that is, $\omega\in[0, \omega_{max}]$.  Therefore, our min-max
  problem \eqref{min_max} reduces to
\begin{equation}\label{min_max_reduced}
  \min_{\alpha}\left(\max_{0\le\omega\le \omega_{max}}|\rho_{n}(\omega,\alpha,-\alpha)|\right).
\end{equation}  
  We observe numerically, see Figure \ref{fig_1_equioscillation},
\begin{figure}
\includegraphics[width=0.49\textwidth,clip]{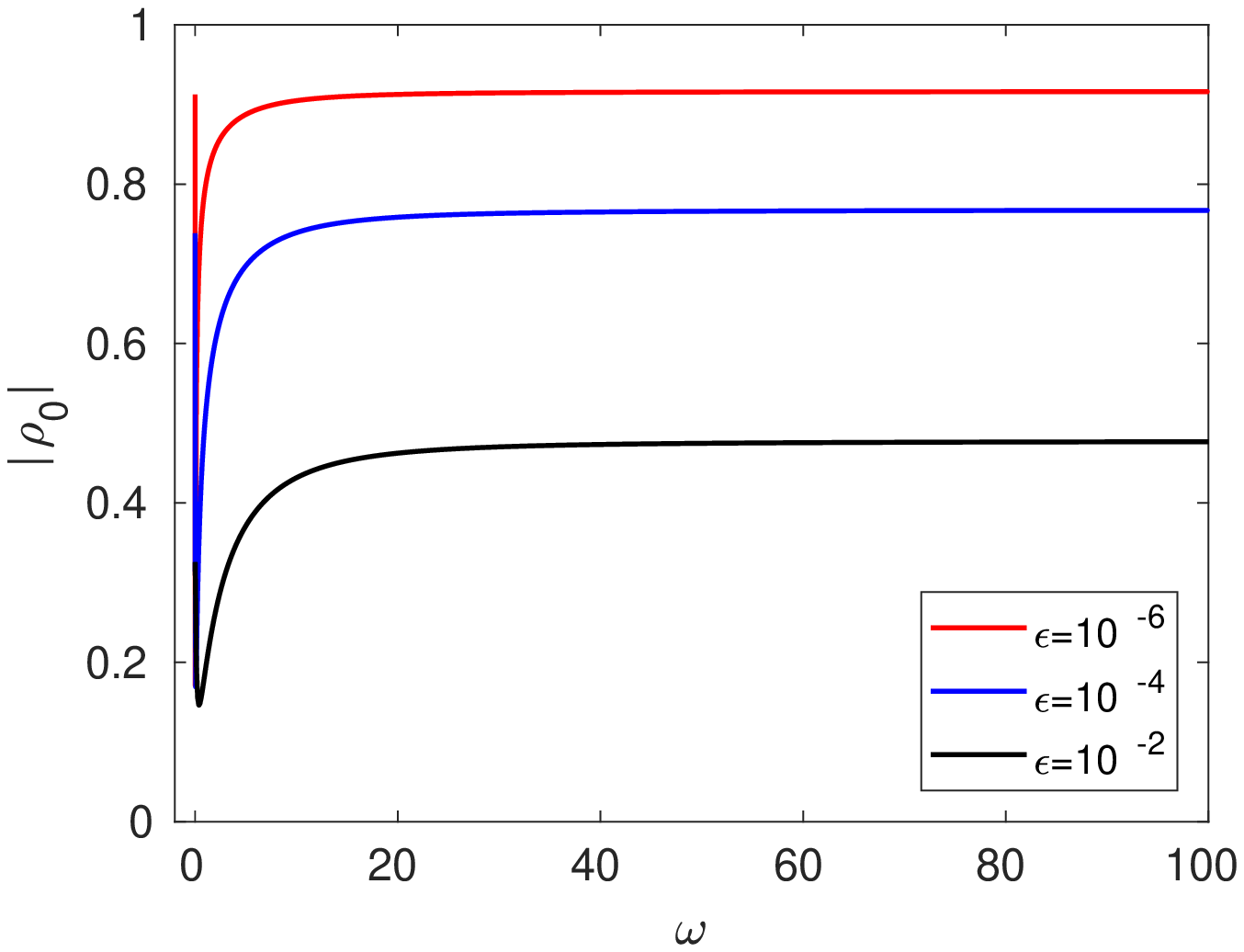}
\includegraphics[width=0.49\textwidth,clip]{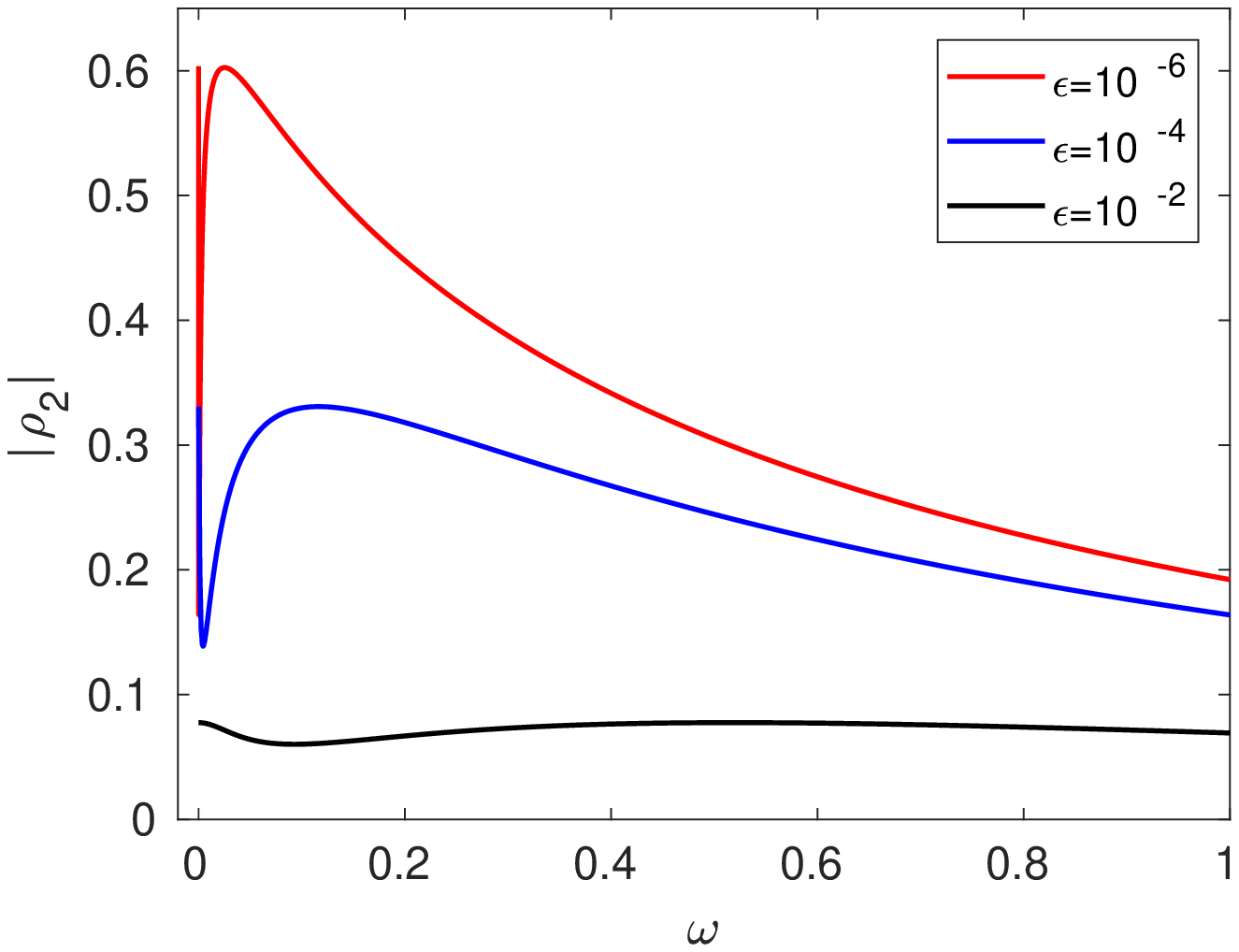}
\caption{Equi-oscillation for different values of $\epsilon$ for $n=0$ (left) and for $n=2$ (right).}
\label{fig_1_equioscillation}
\end{figure}
  that the solution for the min-max problem \eqref{min_max_reduced}
  is given by equi-oscillation. However the behavior of equi-oscillation
  is different for the non-overlapping and overlapping case. We first 
  analyze the non-overlapping case, that is, $n=0$, where the
  equi-oscillation occurs for $\omega=0$ and
  $\omega=\omega_{max}\rightarrow\infty$ (see the left plot of
  Fig. \ref{fig_1_equioscillation}). This means the optimized
  $\alpha$ denoted by $\alpha^{*}_0$ satisfies
\begin{equation}\label{equin0}
  |\rho_{0}(0,\alpha^{*}_0)|=|\rho_{0}(\infty,\alpha^{*}_0)|,
\end{equation}
  where we have dropped the parameter $-\alpha$ for simplicity,
 $\rho_n(\omega,\alpha):=\rho_n(\omega,\alpha,-\alpha)$.

  To start with, we find the explicit expression for
  $|\rho_{0}(\omega,\alpha)|$. In the proof of Lemma \ref{sym}, we
  expressed $\lambda_1=x+iy$, where $x=\frac{(2+\epsilon)a+z_1}{2a}$ and
  $y=\frac{\omega+z_2}{2a}$ and hence
  \begin{equation}\label{rho0}
  |\rho_{0}(\omega,\alpha)|=\frac{|\alpha+1-\lambda_1|^2}{|(1+\alpha)\lambda_{1}-1|^2}=:\frac{A(\omega,\alpha)}{B(\omega,\alpha)},
\end{equation}
where
\begin{eqnarray*}
A(\omega,\alpha)&=&|\alpha+1-\lambda_{1}|^{2}=\alpha^2 + x^2 + 1 + y^2 - 2x - 2x\alpha + 2\alpha, \\
B(\omega,\alpha)&=& |(1+\alpha)\lambda_{1}-1|^2= y^2 + 1 - 2x + 2\alpha y^2 - 2x\alpha + 2\alpha x^2 + x^2 + y^2\alpha^2 + x^2\alpha^2.
\end{eqnarray*}
$A(\omega,\alpha)$ and $B(\omega,\alpha)$ are complicated
  functions of $\omega$ and $\alpha$ which makes the analysis
  difficult. To simplify, we use asymptotic analysis to find an explicit
  expression for $\alpha^{*}_{0}$.  We first express $|\rho_{0}(0,\alpha)|$       and $|\rho_{0}(\infty,\alpha)|$ as polynomials in
  $\epsilon$ using the ansatz $\alpha= C_{\alpha} \epsilon^{\delta}$,
  where $\epsilon=-\frac{b}{a}-2$ and $\delta>0$. The dependence of
  $\alpha^{*}_{0}$ on $\epsilon$ for $n=0$ is illustrated numerically     in the left plot of Fig. \ref{fig_2_alpha_epsilon}.
\begin{figure}
  \includegraphics[width=0.49\textwidth,clip]{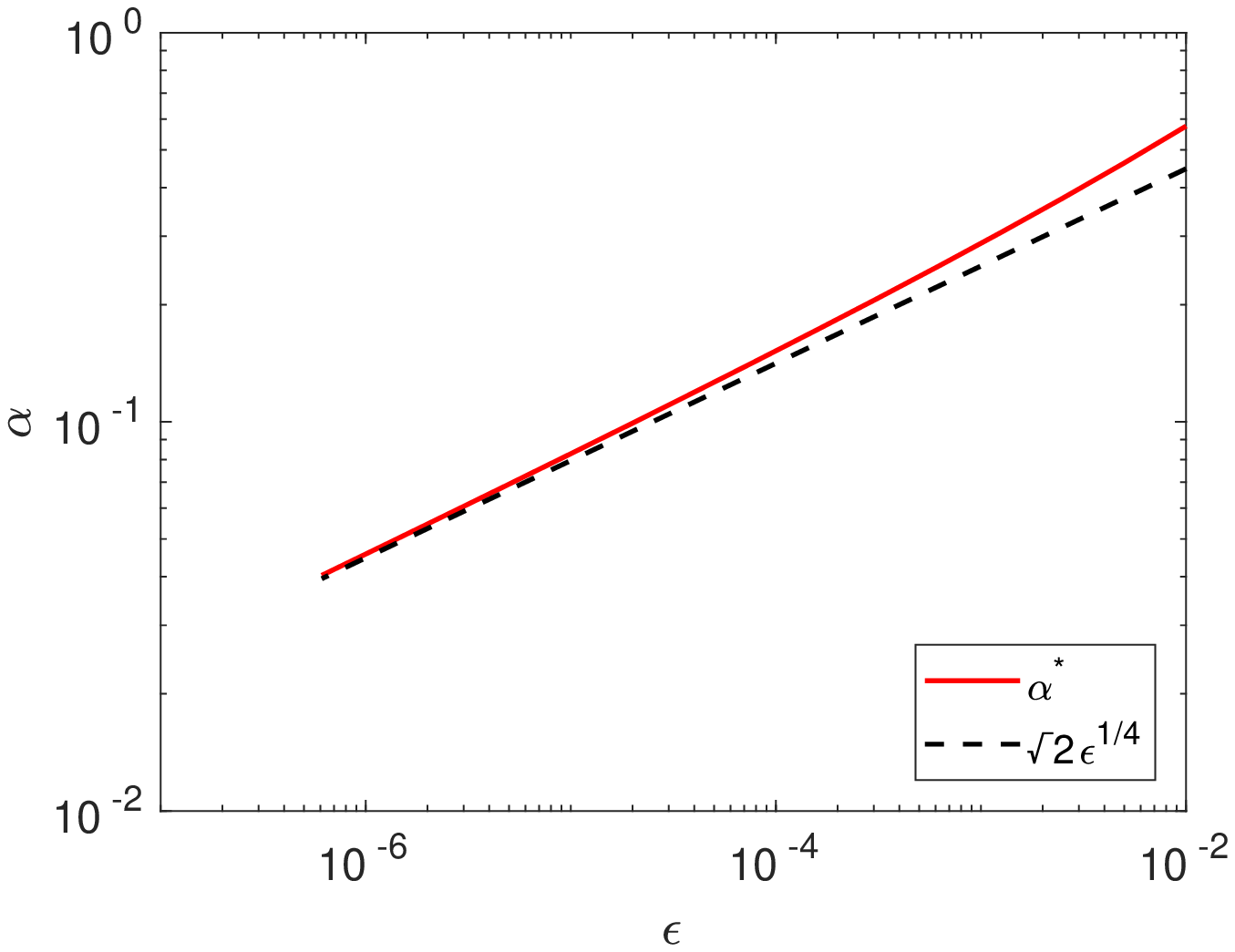}
  \includegraphics[width=0.49\textwidth,clip]{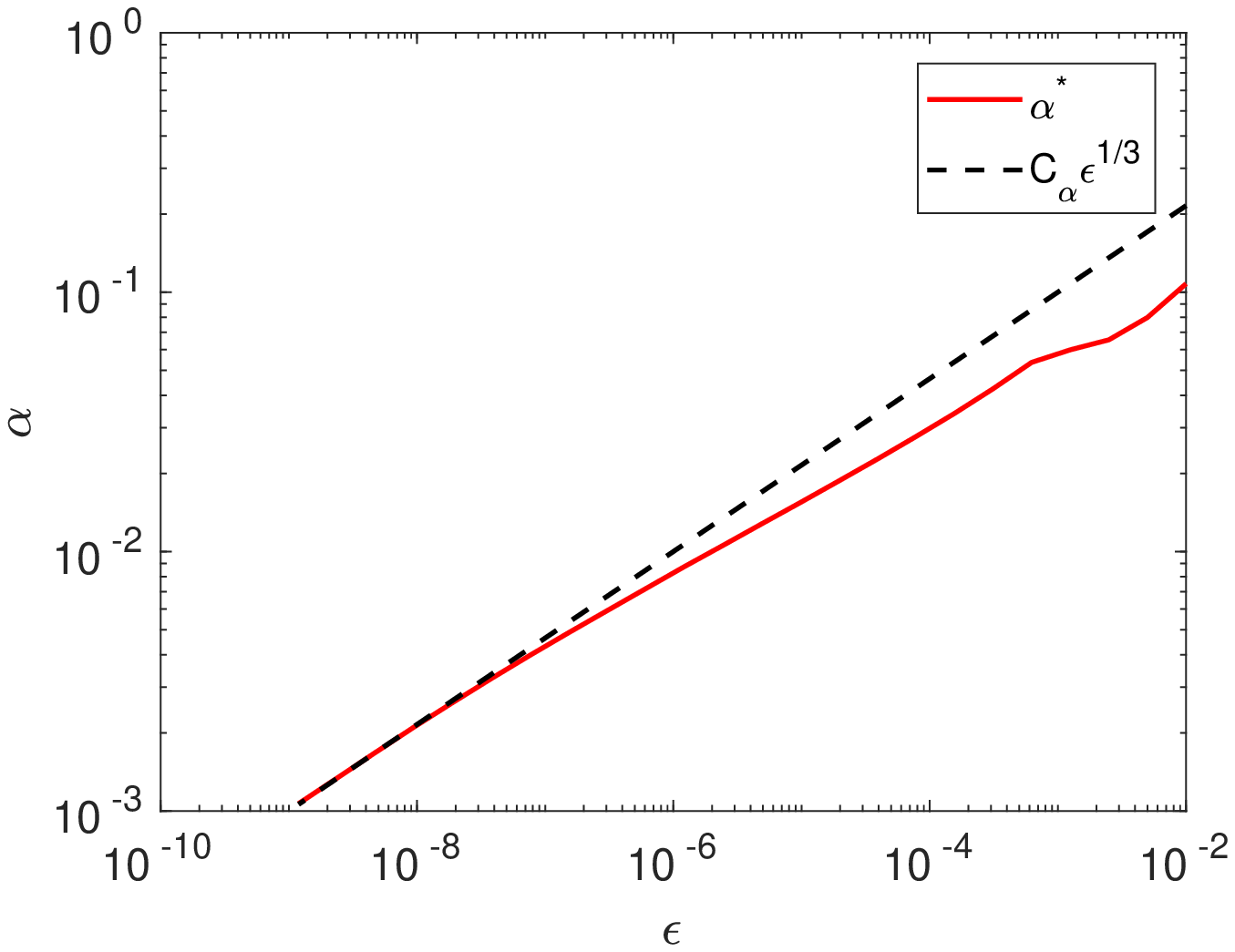}
  \caption{Dependence of $\alpha^{*}$ on $\epsilon$ for $n=0$ (left)
    and with $C_{\alpha}=\left(\frac{1}{n}\right)^{1/3}$ for $n=1$
    (right).}
  \label{fig_2_alpha_epsilon}
\end{figure}

\begin{lemma}\label{poly_rho0}
  For the non-overlapping case, $n=0$, and for small $\epsilon>0$, the
  modulus of the convergence factor $|\rho_{0}(\omega,\alpha)|$ for
  the OWR algorithm for $\omega=0$ and $\omega\rightarrow\infty$ is
  given by
\begin{equation}\label{rho0_0}
|\rho_{0}(0,\alpha)|=1-\frac{4}{C_{\alpha}}\epsilon^{\frac{1}{2}-\delta}+\mathcal{O}(\epsilon^{1-2\delta})
\end{equation}
and
\begin{equation}\label{rho0_inf}
|\rho_{0}(\infty,\alpha)|=1-2 C_{\alpha}\epsilon^{\delta} + \mathcal{O}(\epsilon^{2\delta}).
\end{equation}
\end{lemma}
\begin{proof}
  We first find asymptotic expressions for $x$ and $y$ which are
  functions of $\omega$. From the expression of
  $\lambda_{1}(\omega)=\frac{i\omega -b + \sqrt{(i\omega-
    b)^2-4a^2}}{2a}$, substituting $\omega=0$, $b=-(2+\epsilon)a$ and
using a Taylor expansion for
$(1+\epsilon)^{1/2}=1+\frac{\epsilon}{2}-\frac{\epsilon^2}{8}+\mathcal{O}(\epsilon^3)$ gives
\begin{equation*}
\lambda_{1}(0)=\frac{2+\epsilon}{2}+\frac{\sqrt{4\epsilon+\epsilon^2}}{2}
= 1+\frac{\epsilon}{2}+\epsilon^{1/2}\sqrt{1+\frac{\epsilon}{4}} = 1+ \epsilon^{1/2} +\frac{\epsilon}{2} + \mathcal{O}(\epsilon^{3/2}).
\end{equation*}
  Thus we have $x=1+ \epsilon^{1/2} +\frac{\epsilon}{2} + 
  \mathcal{O}(\epsilon^{3/2})$ and $y=0$. Using the ansatz
  $\alpha=C_{\alpha}\epsilon^{\delta}$, we arrive at
  $A(0,\alpha)=C_{\alpha}^2 \epsilon^{2\delta}-2 C_{\alpha} \epsilon^{\delta    +1/2}+\epsilon
+ \mathcal{O}(\epsilon^{\delta+1})$ and similarly,
$B(0,\alpha)=C_{\alpha}^2 \epsilon^{2 \delta}+2 C_{\alpha} \epsilon^{\delta+1/2}+2 C_{\alpha}^2 \epsilon^{2\delta+1/2}+\mathcal{O}(\epsilon)$.
  Therefore,
\begin{eqnarray*}
|\rho_{0}(0,\alpha)|&=&\frac{A(0,\alpha)}{B(0,\alpha)}=\frac{1-\frac{2}{C_{\alpha}}\epsilon^{1/2-\delta}+\frac{1}{C^{2}_{\alpha}}\epsilon^{1-2\delta}+\mathcal{O}(\epsilon^{1-\delta})}{1+\frac{2}{C_{\alpha}}\epsilon^{1/2-\delta}+2 \epsilon^{1/2}+\mathcal{O}(\epsilon^{1-2\delta})} \\
&=&\left(1-\frac{2}{C_{\alpha}}\epsilon^{1/2-\delta}+\frac{\epsilon^{1-2\delta}}{C^{2}_{\alpha}}+\mathcal{O}(\epsilon^{1-\delta})\right)\left(1-\frac{2}{C_{\alpha}}\epsilon^{1/2-\delta}+\frac{2\epsilon^{1-2\delta}}{C^{2}_{\alpha}}+\mathcal{O}(\epsilon^{1/2})\right) \\
&=& 1-\frac{4}{C_{\alpha}}\epsilon^{1/2-\delta} + \mathcal{O}(\epsilon^{1/2}).
\end{eqnarray*}
  Since $\lim_{\omega\rightarrow\infty}\lambda_1(\omega)=\infty$,
   it is easier to find an expression for $|\rho_{0}(\infty,\alpha)|$
    which can be rewritten as
\begin{eqnarray*}
|\rho_{0}(\infty,\alpha)|&=&\lim_{\omega\rightarrow\infty}\frac{A(\omega,\alpha)}{B(\omega,\alpha)}=\lim_{\omega\rightarrow\infty}\left|\frac{\frac{\alpha+1}{\lambda_1(\omega)}-1}{1+\alpha-\frac{1}{\lambda_1(\omega)}}\right|^2
=\left|\frac{1}{1+\alpha}\right|^2= \frac{1}{1+2 C_{\alpha}\epsilon^{\delta}+C^{2}_{\alpha}\epsilon^{2\delta}}\\
&=&1-2 C_{\alpha}\epsilon^{\delta} + \mathcal{O}(\epsilon^{2\delta}),
\end{eqnarray*}
which concludes the proof.
\end{proof}

\begin{theorem}\label{opt_alpha_n_0}
For the OWR algorithm with no overlap, $n=0$ and for 
small $\epsilon>0$, if $\alpha^{*}_{0}=\sqrt{2}\epsilon^{1/4}$,
 then the convergence factor $\rho_{0}$ satisfies
\begin{equation}
|\rho_{0}(\omega,\alpha)|\le|\rho_{0}(0,\alpha^{*}_{0})|
 \sim 1-2\sqrt{2}\epsilon^{1/4}+\mathcal{O}(\epsilon^{1/2}).
\end{equation}
\end{theorem}
\begin{proof}
  For the non-overlapping case, the solution of the min-max problem
  \eqref{min_max_reduced} is given by equi-oscillation for 
  $\omega=0$ and $\omega\rightarrow\infty$, see the left plot of
  Fig. \ref{fig_1_equioscillation}. Lemma \ref{poly_rho0} provides us
  the expressions for $|\rho_{0}(0,\alpha)|$ and
  $|\rho_{0}(\infty,\alpha)|$ which are equal for $\alpha^{*}_{0}$.
   Comparing the powers of the dominating terms of
  these expressions results in $\frac{1}{2}-\delta=\delta$ which
  implies $\delta=\frac{1}{4}$. Now, equating the coefficients of
  these dominating terms gives $C_{\alpha}=\sqrt{2}$ and this completes
  the proof.
\end{proof}

  The analysis is different for the overlapping case $n>0$. Numerically,
  we observe that the solution of the min-max problem
  \eqref{min_max_reduced} is also given by equi-oscillation, see the
  right plot of Fig. \ref{fig_1_equioscillation}. But in this case,
  equi-oscillation for $|\rho_{n}(\omega,\alpha)|$ occurs for $\omega=0$
  and $\omega=\tilde{\omega}$, where $\tilde{\omega}\rightarrow 0$ as
  $\epsilon\rightarrow 0$. The dependence of the optimized
   $\alpha^{*}$ and $\tilde{\omega}$ on $\epsilon$ for 
   $n=1$ can be seen in the right plot of Fig. \ref{fig_2_alpha_epsilon} 
   and in Fig. \ref{fig_3_omega_epsilon}.
\begin{figure}
\centering
\includegraphics[scale=0.5]{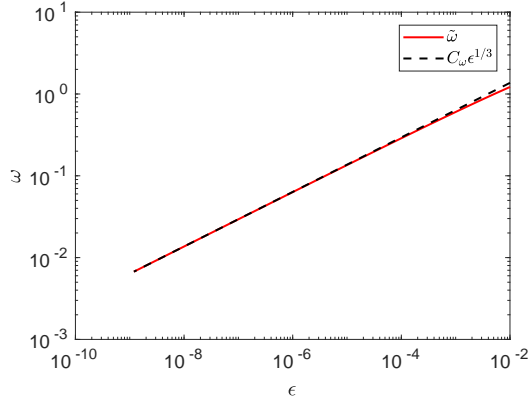}
\caption{Dependence of $\tilde{\omega}$ on $\epsilon$ with $C_{\omega}=\frac{2 a}{n}\left(\frac{1}{n}\right)^{1/3}$.}
\label{fig_3_omega_epsilon}
\end{figure}
  We therefore use the ansatz $\tilde{\omega}:= C_{\omega}  \epsilon^{\eta}$, for some $\eta>0$. Solving the problem
  \eqref{min_max_reduced} is equivalent to solving the system of
  equations
\begin{equation}\label{equi_n}
|\rho_{n}(0,\alpha^{*}_{n})|=|\rho_{n}(\tilde{\omega},\alpha^{*}_{n})| \quad \mbox{and} \quad \frac{\partial}{\partial \omega}|\rho(\tilde{\omega},\alpha^{*}_{n})|=0,
\end{equation}
  where $\alpha^{*}_{n}$ is the optimized $\alpha$ for an overlap of
  size $n$. We first solve the equation $\frac{\partial}{\partial
  \omega}|\rho(\tilde{\omega},\alpha^{*}_{n})|=0$ and find a relation
  between $\tilde{\omega}$ and $\alpha^{*}_{n}$. Substituting this
  relation into the first equation of \eqref{equi_n} then gives an
  explicit expression for $\alpha^{*}_{n}$.

\begin{lemma}\label{deriv}
  For the overlapping case, $n>0$, solving $\frac{\partial}{\partial
  \omega}|\rho(\tilde{\omega},\alpha^{*}_{n})|=0$ gives us the
  relation
\begin{equation}\label{relat}
  \eta=\delta \quad \mbox{and} \quad C_{\omega}=\frac{2 a}{n} C_{\alpha},
\end{equation}
  where $\tilde{\omega}=C_{\omega}\epsilon^{\eta}$ and
  $\alpha^{*}_{n}=C_{\alpha}\epsilon^{\delta}$.
\end{lemma}
\begin{proof}
  We recall the expression for the convergence factor \eqref{cfowr}
   for the OWR algorithm with $n$ overlap,
  \begin{equation*}
|\rho_{n}(\omega,\alpha)|=\left|\left(\frac{\alpha+1-\lambda_1}{(1+\alpha)\lambda_{1}-1}\right)^2 \left(\frac{1}{\lambda_{1}^2}\right)^n\right| = \left|\rho_{0}(\omega,\alpha)\right|\left|\left(\frac{1}{\lambda_{1}^2}\right)^n\right| 
 =\frac{A(\omega,\alpha)}{B(\omega,\alpha)} D(\omega),
\end{equation*}
where $\lambda_1=x+iy$ and
\begin{eqnarray*}
A(\omega,\alpha)&=&|\alpha+1-\lambda_{1}|^{2}=\alpha^2 + x^2 + 1 + y^2 - 2x - 2x \alpha + 2\alpha,\\
B(\omega,\alpha)&=&|(1+\alpha)\lambda_{1}-1|^2= y^2 + 1 - 2x + 2\alpha y^2 - 2x \alpha + 2\alpha x^2 + x^2 + y^2\alpha^2 + x^2\alpha^2, \\
D(\omega)&=& \left|\left(\frac{1}{\lambda_{1}^2}\right)^n\right| = \left(\frac{1}{x^2+y^2}\right)^n. 
\end{eqnarray*}
  We express $A(\omega,\alpha)$, $B(\omega,\alpha)$ and $D(\omega)$ 
  as polynomials in $\omega$ and then find their derivative with
  respect to $\omega$. In the proof of Lemma \ref{sym}, we defined
  $x=\frac{(2+\epsilon)a+z_{1}}{2a}$ and $y=\frac{\omega+z_2}{2a}$,
  where $z_1+i z_2=\sqrt{r+i p}$ and hence
\begin{equation*}
  z^{2}_{1}-z^{2}_2=r \quad \mbox{and}\quad 2 z_1 z_2=p.
\end{equation*}
  Solving these two equations results in solving
  $z_{2}=\frac{p}{2 z_{1}}$ and $4 z^{4}_{1}-4 z^{2}_{1} r-p^2=0$,
  which gives $z^{2}_{1}=\frac{r+\sqrt{r^2+p^2}}{2}$. Let
  $rs:=\sqrt{r^2+p^2}$, where $r:=\epsilon^2 a^2 - \omega^2 +  
  4\epsilon a^2$ and   $p:=2(2+\epsilon)\omega a$. A Taylor expansion for $(1+y)^{1/2}$,  $y<1$, leads for $rs$ to
\begin{eqnarray*}
rs&=&\sqrt{r^2+p^2}=\sqrt{(-\omega^2+4 a^2 \epsilon+ a^2 \epsilon^2)^2+4 \omega^2 a^2 (2+\epsilon)^2}  \\
&=&\sqrt{16 a^2 \omega^2 +\omega^4  +8 \omega^2 a^2 \epsilon+ 16 \epsilon^2 a^4+ \dots}  \\
&=& 4 \omega a \sqrt{1+\left(\frac{\omega^2}{16 a^2}+\frac{\epsilon}{2}+\frac{a^2 \epsilon^2}{\omega^2}+\dots\right)}  \\
&=&4 \omega a +\frac{\omega^3}{8 a}+\omega a \epsilon+\frac{2 a^3 \epsilon^2}{\omega}+ \mathcal{O}(\omega^5). 
\end{eqnarray*}
   Similarly, we obtain expressions for $z_1$ and $z_2$,
\begin{eqnarray*}
z_{1}&=&\sqrt{\frac{r+rs}{2}} \\
&=& \sqrt{2 a \omega} \sqrt{1-\frac{\omega}{4 a}+\frac{\omega^2}{32 a^2}+\frac{a \epsilon}{\omega}+\frac{\epsilon}{4}+\frac{a \epsilon^2}{4 \omega}+ \dots} \\
&=& \sqrt{2 a \omega}-\frac{\sqrt{2} \omega^{3/2}}{8 \sqrt{a}}+\frac{\sqrt{2} \omega^{5/2}}{128 a^{3/2}}+\frac{\sqrt{2} a^{3/2} \epsilon}{2 \sqrt{\omega}}+ \mathcal{O}(w^{1/2}\epsilon),
\end{eqnarray*}
 and $z_2=\frac{p}{2 z_{1}}=\sqrt{2 a \omega}+\frac{\sqrt{2}
  \omega^{3/2}}{8 \sqrt{a}}-\frac{\sqrt{2} a^{3/2} \epsilon}{2
   {\sqrt{\omega}}}+\mathcal{O}(w^{1/2}\epsilon)$.
   Substituting these into the expressions for $x$ and $y$ leads to
\begin{eqnarray*}
x&=&1+\frac{\sqrt{2} \sqrt{\omega}}{2 \sqrt{a}}-\frac{\sqrt{2} \omega^{3/2}}{16 a^{3/2}}+\mathcal{O}(\omega^{5/2}), \\
y&=&\frac{\sqrt{2} \sqrt{\omega}}{2 \sqrt{a}}+\frac{\omega}{2 a}+\frac{\sqrt{2} \omega^{3/2}}{16 a^{3/2}}+\mathcal{O}(\omega^{5/2}). 
\end{eqnarray*}
   Collecting the dominating terms gives us asymptotic expressions
   for $A(\omega,\alpha)$ and $B(\omega,\alpha)$,
\begin{eqnarray*}
A(\omega,\alpha)&=&\frac{\omega}{a}-\frac{\sqrt{2} C_{\alpha} \epsilon^{\delta}\sqrt{\omega}}{\sqrt{a}}+\frac{\sqrt{2} \omega^{3/2}}{2 a^{3/2}}+ \mathcal{O}(\omega^{2}), \\
B(\omega,\alpha)&=&\frac{\omega}{a}+\frac{\sqrt{2} C_{\alpha} \epsilon^{\delta}\sqrt{\omega}}{\sqrt{a}}+\frac{\sqrt{2} \omega^{3/2}}{2 a^{3/2}}+ \mathcal{O}(\omega^{2}).
\end{eqnarray*}
    To find an asymptotic expression for $D(\omega)$, we use the Taylor
    expansion $\left(\frac{1}{1+y}\right)^n=1-ny+\frac{n^2 y^2}{2}+
   \mathcal{O}(y^3)$ and obtain
\begin{eqnarray}
D(\omega)&=&\left(\frac{1}{x^2+y^2}\right)^n 
=\left(\frac{1}{1+\frac{\sqrt{2} \sqrt{\omega}}{\sqrt{a}}+\frac{\omega}{a}+\frac{3 \sqrt{2} \omega^{3/2}}{8 a^{3/2}}+\mathcal{O}(\omega^{2})}\right)^n  \nonumber \\
&=& 1-\frac{n \sqrt{2} \sqrt{\omega}}{\sqrt{a}}-\frac{n \omega}{a}+\frac{n^2 \omega}{a}+\mathcal{O}(\omega^{3/2}). \label{asyms_D}
\end{eqnarray}
  Further, differentiating $|\rho_{n}(\omega,\alpha)|=\frac{A
  (\omega,\alpha)}{B(\omega,\alpha)} D(\omega)$ with respect
  to $\omega$ produces 
  \[\frac{\partial}{\partial \omega}|\rho_{n}(\omega,\alpha)|=\frac{
B(\omega,\alpha)[A(\omega,\alpha) D_{\omega}(\omega)+A_{\omega}(\omega,\alpha)D(\omega)]-A(\omega,\alpha)D(\omega) B_{\omega}(\omega,\alpha)}{B^2(\omega,\alpha)}.\] 
 Let $F_1(\omega,\alpha):=B(\omega,\alpha)[A(\omega,\alpha)
   D_{\omega}(\omega)+A_{\omega}(\omega,\alpha)D(\omega)]$
   and $F_2(\omega,\alpha):=A(\omega,\alpha)D(\omega) 
   B_{\omega}(\omega,\alpha)$, where $A_{\omega}(\omega,\alpha)$
   denotes the partial derivative of $A(\omega,\alpha)$ with respect
   to $\omega$. Collecting the dominating terms in the asymptotic
   expansions of $F_{1}$, $F_{2}$ results in
\begin{eqnarray*}
F_{1}(\omega,\alpha)&=&\frac{\sqrt{2} C_{\alpha} \epsilon^{\delta} \sqrt{\omega}}{2 a^{3/2}}+\frac{\omega}{a^2}+\frac{\sqrt{2} \omega^{3/2}}{2 a^{5/2}}\left(\frac{5}{2}-3n\right)+\mathcal{O}(\omega^2), \\
F_{2}(\omega,\alpha)&=&-\frac{\sqrt{2} C_{\alpha} \epsilon^{\delta} \sqrt{\omega}}{2 a^{3/2}}+\frac{\omega}{a^2}+\frac{\sqrt{2} \omega^{3/2}}{a^{5/2}}\left(\frac{5}{4}-n\right)+\mathcal{O}(\omega^2).
\end{eqnarray*}
   Equating $F_1(\tilde{\omega},\alpha)-F_{2}(\tilde{\omega},\alpha)=0$,
   we obtain $\tilde{\omega}=\frac{2 a C_{\alpha}}{n}
   \epsilon^{\delta}$. Since we use the ansatz $\tilde{\omega}=
   C_{\omega} \epsilon^{\eta}$, comparing the exponents and coefficients
   simplifies to \eqref{relat} and this completes the proof.
\end{proof}

  Now we need to solve the first equation of \eqref{equi_n},
  $|\rho_{n}(0,\alpha^{*}_{n})|=|\rho_{n}(\tilde{\omega},\alpha^{*}_{n})|$.  
  We do this in a similar way as for the non-overlapping case, $n=0$.
  We express $|\rho_{n}(0,\alpha^{*}_{n})|$ and $|\rho_{n}\tilde{\omega},
  \alpha^{*}_{n})|$ as asymptotic expansions in $\epsilon$.

\begin{lemma}\label{poly_n}
  For the overlapping case, $n>0$, the modulus of the convergence
  factor $|\rho_{n}(\omega,\alpha)|$ for the OWR algorithm for
  $\omega=0$ and $\omega=\tilde{\omega}$ is for small $\epsilon$ given
  by
\begin{equation}\label{rhon_0}
|\rho_{n}(0,\alpha)|=1-\frac{4}{C_{\alpha}} \epsilon^{\frac{1}{2}-\delta}+\mathcal{O}(\epsilon^{1-2\delta}),
\end{equation}
and 
\begin{equation}\label{rhon_tilde}
|\rho_{n}(\tilde{\omega},\alpha)|=1-\frac{2\sqrt{2} C_{\alpha} \sqrt{a} \epsilon^{\delta/2}}{\sqrt{C_{\omega}}}-\frac{n \sqrt{2} \sqrt{C_{\omega}} \epsilon^{\delta/2}}{\sqrt{a}}+\mathcal{O}(\epsilon^{\delta}).
\end{equation}
\end{lemma}
\begin{proof}
  From the polynomial expansion of $D(\omega)$ given in \eqref{asyms_D},
  we obtain $D(0)=\left|\left(\frac{1}{\lambda^{2}_1(0)}\right)^n
  \right|=1$. Substituting this into the formula for the convergence 
  factor for OWR with $n$ overlap leads to
  $\left|\rho_{n}(0,\alpha)\right|=\left|\rho_{0}(0,\alpha)\right|
  \left|\left(\frac{1}{\lambda^{2}_1(0)}\right)^n\right|
  =\left|\rho_{0}(0,\alpha)\right|$ and we arrive at \eqref{rhon_0}.

  The analysis to find the expression for $|\rho_{n}(\tilde{\omega},
  \alpha)|$ is similar. Substituting $\tilde{\omega}=C_{\omega}
  \epsilon^{\delta}$ in to the  expressions for $A(\omega,\alpha)$ 
  and  $B(\omega,\alpha)$ leads to 
  \begin{eqnarray*}
A(\tilde{\omega},\alpha)&=&\frac{\tilde{\omega}}{a}-\frac{\sqrt{2} C_{\alpha} \epsilon^{\delta}\sqrt{\tilde{\omega}}}{\sqrt{a}}+\frac{\sqrt{2} \tilde{\omega}^{3/2}}{2 a^{3/2}}+ \mathcal{O}(\tilde{\omega}^2) 
= \frac{C_{\omega} \epsilon^{\delta}}{a} -\frac{\sqrt{2} C_{\alpha} \sqrt{C_{\omega}} \epsilon^{3\delta/2}}{\sqrt{a}}+\frac{C_{\omega}^{3/2} \epsilon^{3\delta/2}}{\sqrt{2} a^{3/2}} +\mathcal{O}(\epsilon^{2 \delta})  \\
&=& \frac{C_{\omega} \epsilon^{\delta}}{a}\left(1-\frac{\sqrt{2} C_{\alpha} \sqrt{a} \epsilon^{\delta/2}}{\sqrt{C_{\omega}}} +\frac{\sqrt{C_{\omega}} \epsilon^{\delta/2}}{\sqrt{2} \sqrt{a}}+\mathcal{O}(\epsilon^\delta)\right),     \\
B(\tilde{\omega},\alpha)&=&\frac{\tilde{\omega}}{a}+\frac{\sqrt{2} C_{\alpha} \epsilon^{\delta}\sqrt{\tilde{\omega}}}{\sqrt{a}}+\frac{\sqrt{2} \tilde{\omega}^{3/2}}{2 a^{3/2}}+ \mathcal{O}(\tilde{\omega}^2)  
= \frac{C_{\omega} \epsilon^{\delta}}{a} +\frac{\sqrt{2} C_{\alpha} \sqrt{C_{\omega}} \epsilon^{3\delta/2}}{\sqrt{a}}+\frac{C_{\omega}^{3/2} \epsilon^{3\delta/2}}{\sqrt{2} a^{3/2}} +\mathcal{O}(\epsilon^{5 \delta/2})  \\
&=& \frac{C_{\omega} \epsilon^{\delta}}{a}\left(1+\frac{\sqrt{2} C_{\alpha} \sqrt{a} \epsilon^{\delta/2}}{\sqrt{C_{\omega}}} +\frac{\sqrt{C_{\omega}} \epsilon^{\delta/2}}{\sqrt{2} \sqrt{a}}+\mathcal{O}(\epsilon^{\delta})\right).          
\end{eqnarray*}
  Dividing $A(\tilde{\omega},\alpha)$ by $B(\tilde{\omega},\alpha)$ 
  and using the Taylor expansion $\frac{1}{1+z}=1-z+\frac{z^2}{2}
  +\mathcal{O}(z^3) $ gives us 
\begin{eqnarray*}
\frac{A(\tilde{\omega},\alpha)}{B(\tilde{\omega},\alpha)}&=&\left(1-\frac{\sqrt{2} C_{\alpha} \sqrt{a} \epsilon^{\delta/2}}{\sqrt{C_{\omega}}} +\frac{\sqrt{C_{\omega}} \epsilon^{\delta/2}}{\sqrt{2} \sqrt{a}}+ \dots\right)\left(1-\frac{\sqrt{2} C_{\alpha} \sqrt{a} \epsilon^{\delta/2}}{\sqrt{C_{\omega}}} -\frac{\sqrt{C_{\omega}} \epsilon^{\delta/2}}{\sqrt{2} \sqrt{a}}+ \dots\right) \\
&=& 1-\frac{2\sqrt{2} C_{\alpha} \sqrt{a} \epsilon^{\delta/2}}{\sqrt{C_{\omega}}}-\frac{C_{\omega} \epsilon^{\delta}}{2 a}+\frac{2C^{2}_{\alpha}a\epsilon^{\delta}}{C_\omega}+\mathcal{O}(\epsilon^{3 \delta/2}).
\end{eqnarray*}
  Similarly, we obtain the expression for $D(\tilde{\omega})$,
\[D(\tilde{\omega})=1-\frac{n \sqrt{2} \sqrt{C_{\omega}} \epsilon^{\delta/2}}{\sqrt{a}}-\frac{n C_{\omega} \epsilon^{\delta}}{a}+\frac{n^2 C_{\omega} \epsilon^{\delta}}{a}+\mathcal{O}(\epsilon^{3 \delta/2}).\]
  Multiplying the above expression by the expansion for 
  $\frac{A(\tilde{\omega})}{B(\tilde{\omega})}$, we find,
\begin{eqnarray*}
|\rho_{n}(\tilde{\omega},\alpha)|&=& \left(1-\frac{2\sqrt{2} C_{\alpha} \sqrt{a} \epsilon^{\delta/2}}{\sqrt{C_{\omega}}}-\frac{C_{\omega} \epsilon^{\delta}}{2 a}+\frac{2C^{2}_{\alpha}a\epsilon^{\delta}}{C_\omega}+\mathcal{O}(\epsilon^{3 \delta/2})\right)\left(1-\frac{n \sqrt{2} \sqrt{C_{\omega}} \epsilon^{\delta/2}}{\sqrt{a}}-\frac{n  C_{\omega} \epsilon^{\delta}}{a}+\mathcal{O}(\epsilon^{3 \delta/2})\right) \\
&=& 1-\frac{2\sqrt{2} C_{\alpha} \sqrt{a} \epsilon^{\delta/2}}{\sqrt{C_{\omega}}}-\frac{n \sqrt{2} \sqrt{C_{\omega}} \epsilon^{\delta/2}}{\sqrt{a}}+\mathcal{O}(\epsilon^{\delta}),
\end{eqnarray*}
which completes the proof.
\end{proof}

  We are now ready to prove a remarkably simple formula for the
  optimized parameter depending on the overlap $n>0$ which is
  quite different from the non-overlapping case:
\begin{theorem}\label{opt_alpha_n}
For the OWR algorithm in the overlapping case, $n>0$,
 and for small $\epsilon>0$, if  $\alpha^{*}_{n}= \left(\frac{\epsilon}{n}\right)^{1/3}$, then the convergence factor $\rho_{n}$ satisfies
\begin{equation}
|\rho_{n}(\omega,\alpha)|\le |\rho_{n}(0,\alpha^{*}_{n})|\sim 1-4 n^{1/3} \epsilon^{1/6} + \mathcal{O}(\epsilon^{1/3}).
\end{equation}
\end{theorem}
\begin{proof}
  Since the solution of our min-max problem 
  \eqref{min_max_reduced} is obtained numerically by equi-oscillation
  for $\omega=0$ and $\omega=\tilde{\omega}$, equating the 
  expansions for $|\rho_{n}(0,\alpha)|$ and $|\rho_{n}(\tilde{\omega}
  ,\alpha)|$ given by \eqref{rhon_0} and \eqref{rhon_tilde} and 
  comparing their dominating terms, we obtain 
  \begin{equation*}
\frac{\delta}{2}=\frac{1}{2}-\delta \quad \mbox{and} \quad \frac{4}{C_{\alpha}}=\frac{2\sqrt{2} C_{\alpha} \sqrt{a}}{\sqrt{C_{\omega}}}
+\frac{n \sqrt{2} \sqrt{C_{\omega}}}{\sqrt{a}}.
\end{equation*}
  The first equation  implies that  $\delta=1/3$, and
  substituting the expression for $C_{\omega}$ given by 
  \eqref{relat} results in $C_{\alpha}=\left(\frac{1}{n}
  \right)^{1/3}$. This completes the proof of the theorem.
\end{proof}  
  
\section{Multiple Sub-circuits}\label{Sec:ManySubCircuits}

  We studied so far only the decomposition of the system of
  equations \eqref{ce1} into two sub-systems. However, for the
  practical use on parallel computers, we need to split the system
  into many sub-systems and apply the WR or OWR algorithm to them.
  We thus split the RC circuit of infinite length in
  Fig. \ref{RCinf} into $N_s$ sub-circuits which are denoted by
  $\textbf{v}_{r}$, $r=1,2,\ldots,N_s$. Assume that each
  sub-circuit $\textbf{v}_{r}$ contains $M_r$ nodes. Applying the
  Optimized Waveform Relaxation algorithm (with $\alpha$, $\beta$ as the
  optimization parameters) with $n$ nodes overlap leads to
\begin{equation}\label{multipl_owr}
\begin{array}{l}
\dot{\textbf{v}}_1^{k+1}(t) = \begin{bmatrix}
 			\ddots & \ddots  &\ddots &    \\
			       &  a      & b     & a  \\
			       &         & a     & b+\frac{a}{\alpha+1} 
		      \end{bmatrix}\textbf{v}_{1}^{k+1}(t)  +\begin{bmatrix}
                    \vdots \\
                    0 \\
                   a v_{2,n+1}^{k}(t)-\frac{a}{\alpha+1}v_{2,n}^{k}(t)
                    \end{bmatrix},\\                    
\dot{\textbf{v}}_r^{k+1}(t) = \begin{bmatrix}
			b-\frac{a}{\beta-1}      & a      &       &     \\
			a      & b      & a     &    \\
			       & \ddots & \ddots & \ddots  \\
			       &        &      a & b +\frac{a}{\alpha+1}    
		      \end{bmatrix} \textbf{v}_{r}^{k+1}(t)   +\begin{bmatrix}
                    a v_{r-1,M_r}^{k}(t)+\frac{a}{\beta-1}v_{r-1,M_r+1}^{k}(t)\\
                    0\\
                    \vdots \\
                    0\\
                    a v_{r+1,n+1}^{k}(t)-\frac{a}{\alpha+1}v_{r+1,n}^{k}(t)\end{bmatrix},            \\
\dot{\textbf{v}}_{N_s}^{k+1}(t) = \begin{bmatrix}
			b-\frac{a}{\beta-1}      & a     &       &    \\
			a      & b      & a     &  \\
			       & \ddots       & \ddots     & \ddots 
		      \end{bmatrix}\textbf{v}_{N_s}^{k+1}(t)   +\begin{bmatrix}
                     a v_{N_s-1,M_r}^{k}(t)+\frac{a}{\beta-1}v_{N_s-1,M_r+1}^{k}(t) \\
                    0\\
                    \vdots 
                    \end{bmatrix},                 
\end{array}
\end{equation}
  where $r=2,3,\dots,N_s-1 $ and we have considered the source term
 $\textbf{f}=0$ for simplicity. For $\alpha=\infty$, $\beta=-\infty$,
 we obtain the classical Waveform Relaxation algorithm.

  In the case of OWR for two sub-domains, we saw that with the
  optimal transmission conditions \eqref{opt_alpha}, convergence is
  achieved in 2 iterations and the parameters $\alpha$, $\beta$ represent
  non-local operators in time. In the case of $N_s$ sub-circuits,
  with the use of optimal transmission conditions, one can show
  that convergence can be achieved in $N_s$ iterations
  \cite{RC_small_GA}, and this result still holds in the overlapping
  case. Again these parameters $\alpha$, $\beta$ are non-local
   operators in time and the analysis becomes very
   complicated. Hence we will use $\alpha^{*}_{0}$ and
  $\alpha^{*}_{n}$ given 
 in Theorem \ref{opt_alpha_n_0} and Theorem \ref{opt_alpha_n} 
    from our two sub-circuit analysis and test their
   performance numerically for the multiple sub-circuit case.

\section{Numerical Results}\label{Sec:Numerical}

  We consider an RC circuit with $R=0.5 k\si{\ohm}$, $C=0.63  pF$,
  $a=\frac{1}{RC}$ and $b=-(2+a)\epsilon$ with $\epsilon>0$,
  $\epsilon\rightarrow 0$. Our analysis for both WR and OWR was
  performed assuming that the length of the circuit is infinite. For
  our numerical experiments, we consider a circuit of length
  $N=200$, a homogeneous source, zero initial conditions, and apply
  the backward Euler scheme for time integration with $\Delta
  t=0.1$ and $T=2000$. We use random initial guesses to start both the
  WR and OWR algorithm. The left plot of Fig. \ref{fig_6_wr_vs_owr}
\begin{figure}
\includegraphics[width=0.49\textwidth,clip]{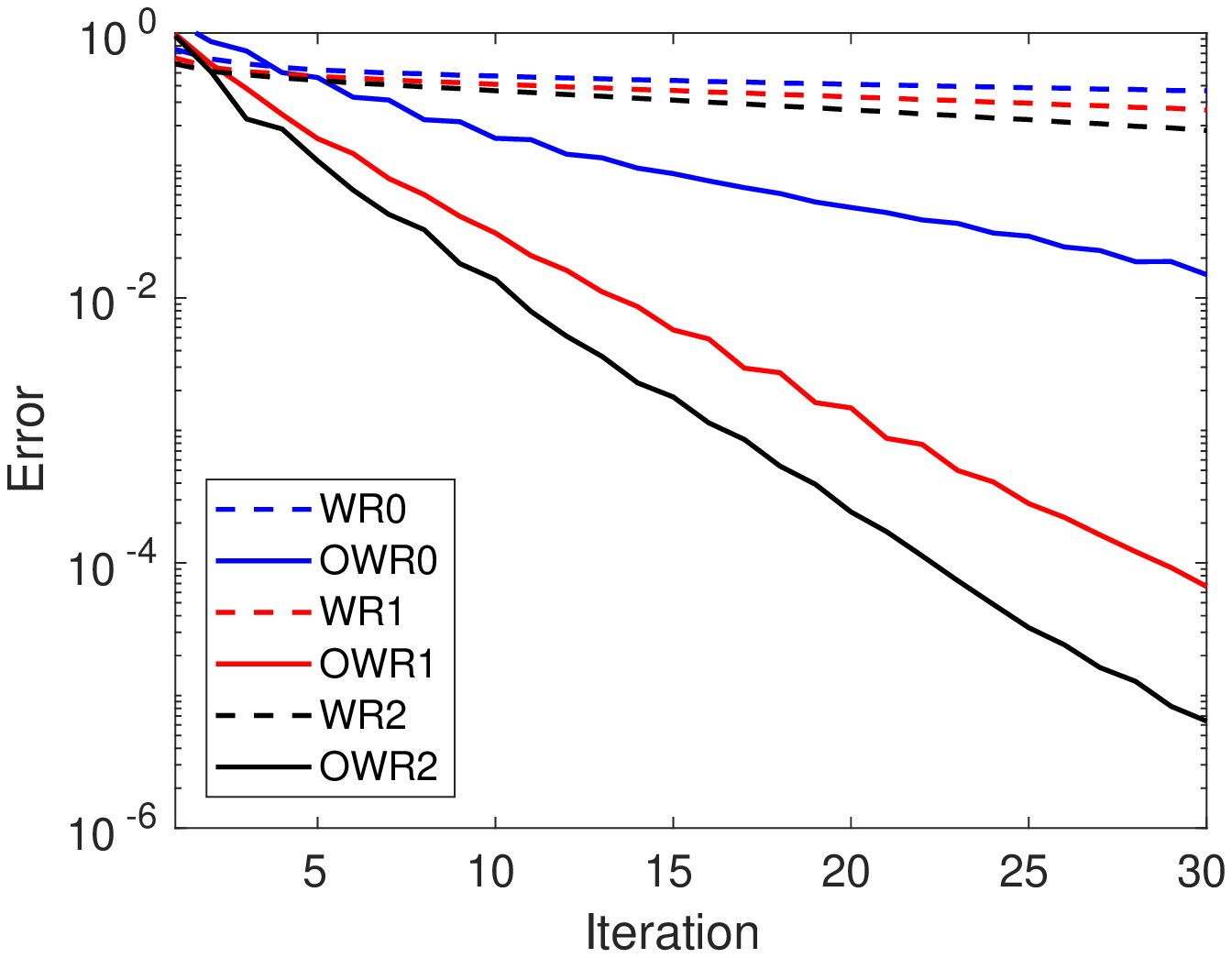}
\includegraphics[width=0.49\textwidth,clip]{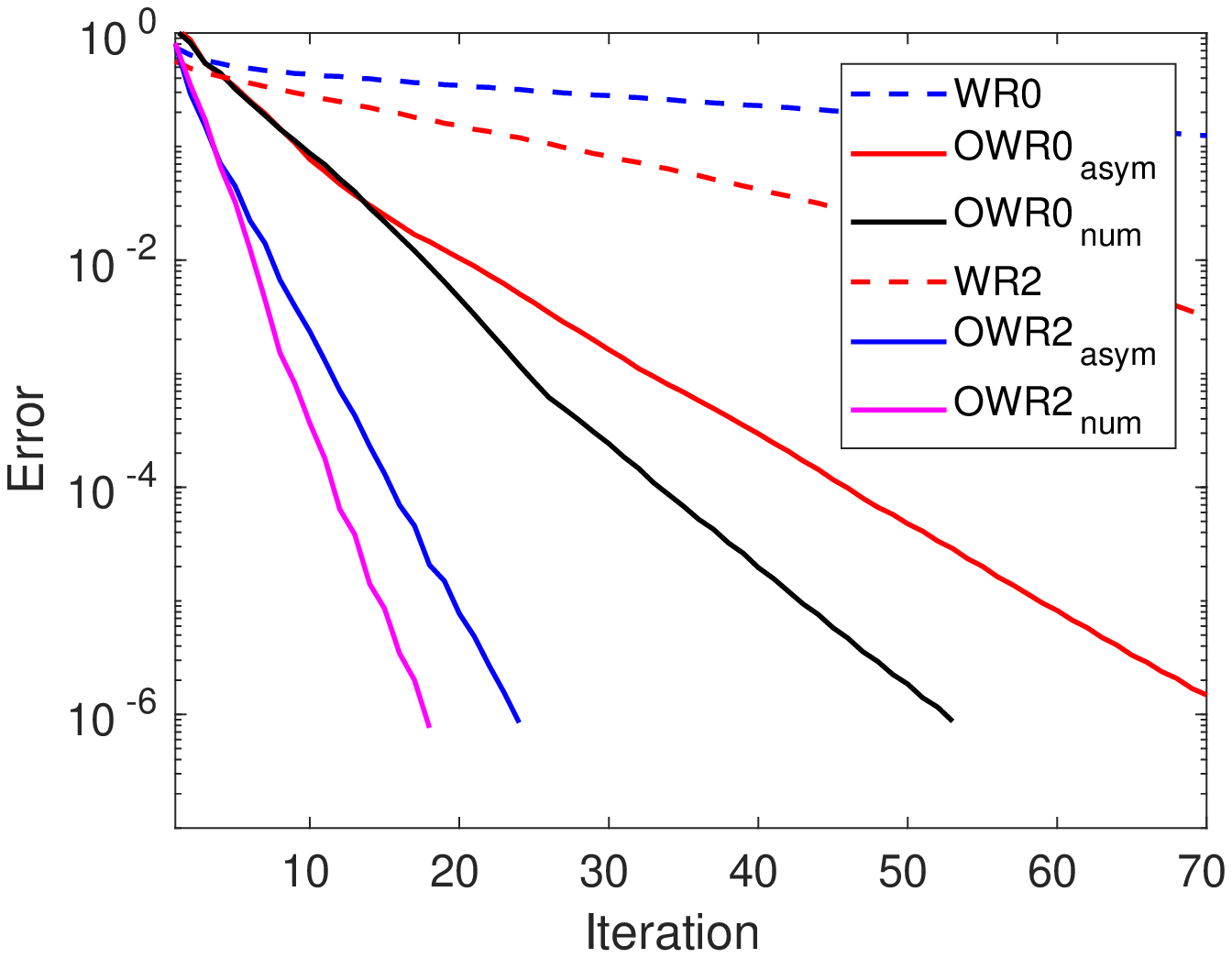}
\caption{Convergence for $T=2000$ for the RC circuit (left) and for
  the heat equation (right).}
\label{fig_6_wr_vs_owr}
\end{figure}
  clearly shows the important influence of our optimized
  transmission conditions \eqref{otc} on the convergence of the
  algorithm. We used the optimized $\alpha$ for OWR given 
  in Theorem \ref{opt_alpha_n_0} and Theorem \ref{opt_alpha_n}. 
  We also illustrate in the same figure the influence of overlap on
  convergence. Increasing the overlap increases the convergence 
  rate in both the WR and OWR algorithm, but the impact of the
  optimized transmission conditions is much more important
  for performance than the influence of the overlap.

  We have seen in \eqref{HeatEquationDiscrete} that for
  $\epsilon\rightarrow0$, the RC circuit of infinite length is an
  approximation to the one dimensional heat equation. We now
  apply the WR and OWR algorithms with different overlaps
  $n=0$ and $n=2$ to the heat equation \eqref{HeatEquationDiscrete}.
  For OWR we need the values of $\alpha^{*}_{0}$ and $\alpha^{*}_{2}$,
  which we obtain both by solving the min-max problem numerically,
  and by using the asymptotic expression 
  in Theorem \ref{opt_alpha_n_0} and Theorem \ref{opt_alpha_n}  
  respectively from our circuit analysis for small
  $\epsilon=10^{-4}$. From the right plot of Fig.
  \ref{fig_6_wr_vs_owr}, we observe that OWR converges much
  faster than WR. Also, the numerically calculated $\alpha^{*}_{0}$
  and $\alpha^{*}_{2}$ give us the best convergence. For overlap
  $n=2$, the convergence of OWR using our asymptotic expression 
  for $\alpha^{*}_{2}$ for $\epsilon=10^{-4}$ needs only 4 more
  iterations, than the best possible choice computed numerically.
  This is very little compared to the iterations and time
  required to compute $\alpha^{*}_{2}$ numerically. 
  This shows that our expressions 
in Theorem \ref{opt_alpha_n_0} and Theorem \ref{opt_alpha_n}
  can also be used when OWR is applied to the heat equation.

We next compare the numerically and asymptotically optimized
  values of $\alpha^{*}_{n}$. For the non-overlapping case $n=0$
 and $\epsilon=10^{-4}$, we show in Fig. \ref{fig_4_optimal_0}
\begin{figure}
\includegraphics[width=0.49\textwidth,clip]{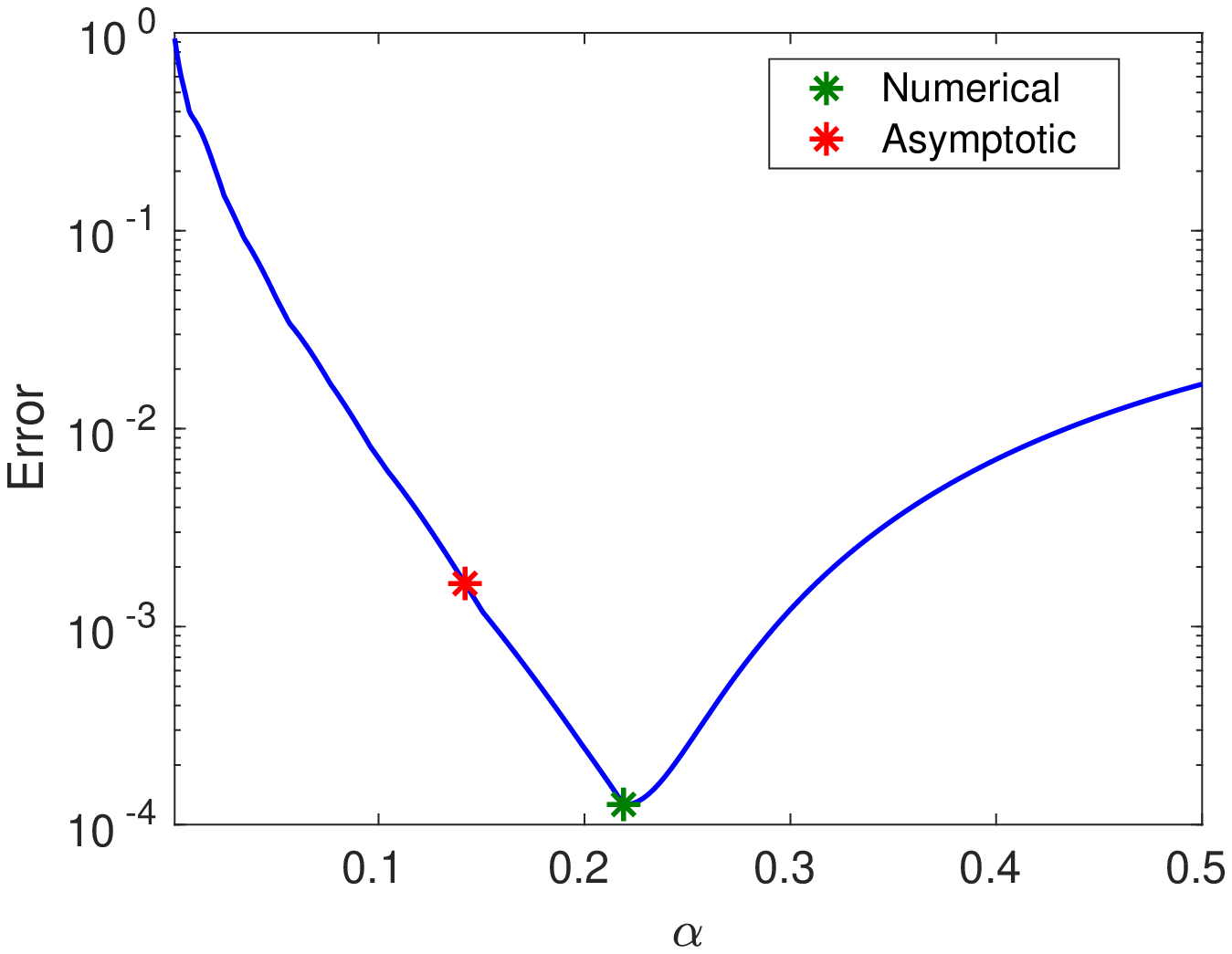}
\includegraphics[width=0.49\textwidth,clip]{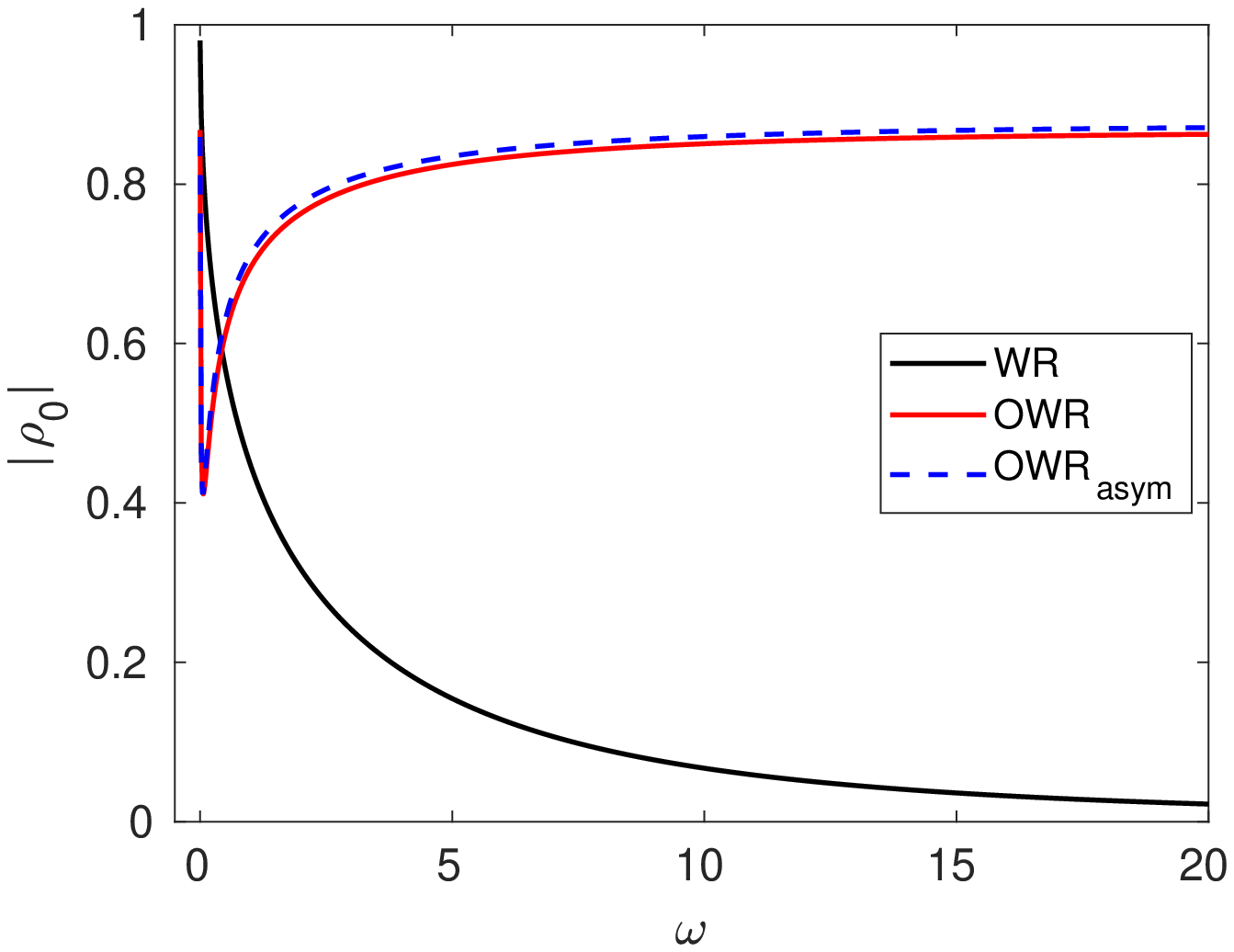}
\caption{Comparison of the optimized $\alpha$ for $\epsilon=10^{-4}$ for $n=0$ (left) and convergence factor for $n=0$ in Laplace space for WR and OWR (right).}
\label{fig_4_optimal_0}
\end{figure}
  on the left the error reduction for different values of $\alpha$,
  together with the numerically and asymptotically optimized value. We
  see that our asymptotic formula underestimates the optimal choice a
  bit. However, the right plot in Fig. \ref{fig_4_optimal_0}
  shows that for the two different values of the optimized
  $\alpha$, the corresponding convergence factors $|\rho_{0}|$ in
   Laplace space are very close to each other, and we also show the
  improvement in the overall convergence for OWR compared to WR. 
  In Fig. \ref{fig_4_optimal_n}
\begin{figure}
\includegraphics[width=0.49\textwidth,clip]{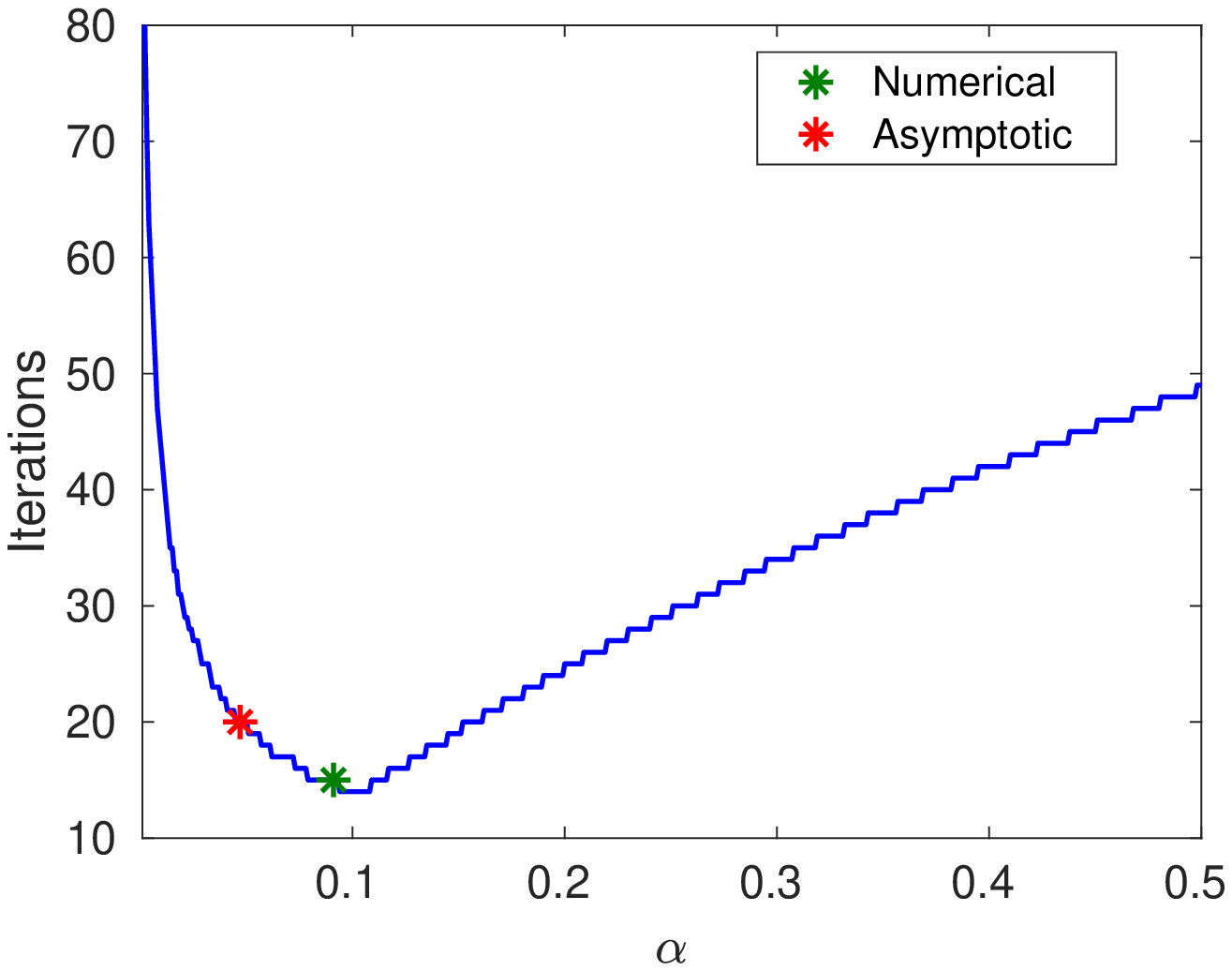}
\includegraphics[width=0.49\textwidth,clip]{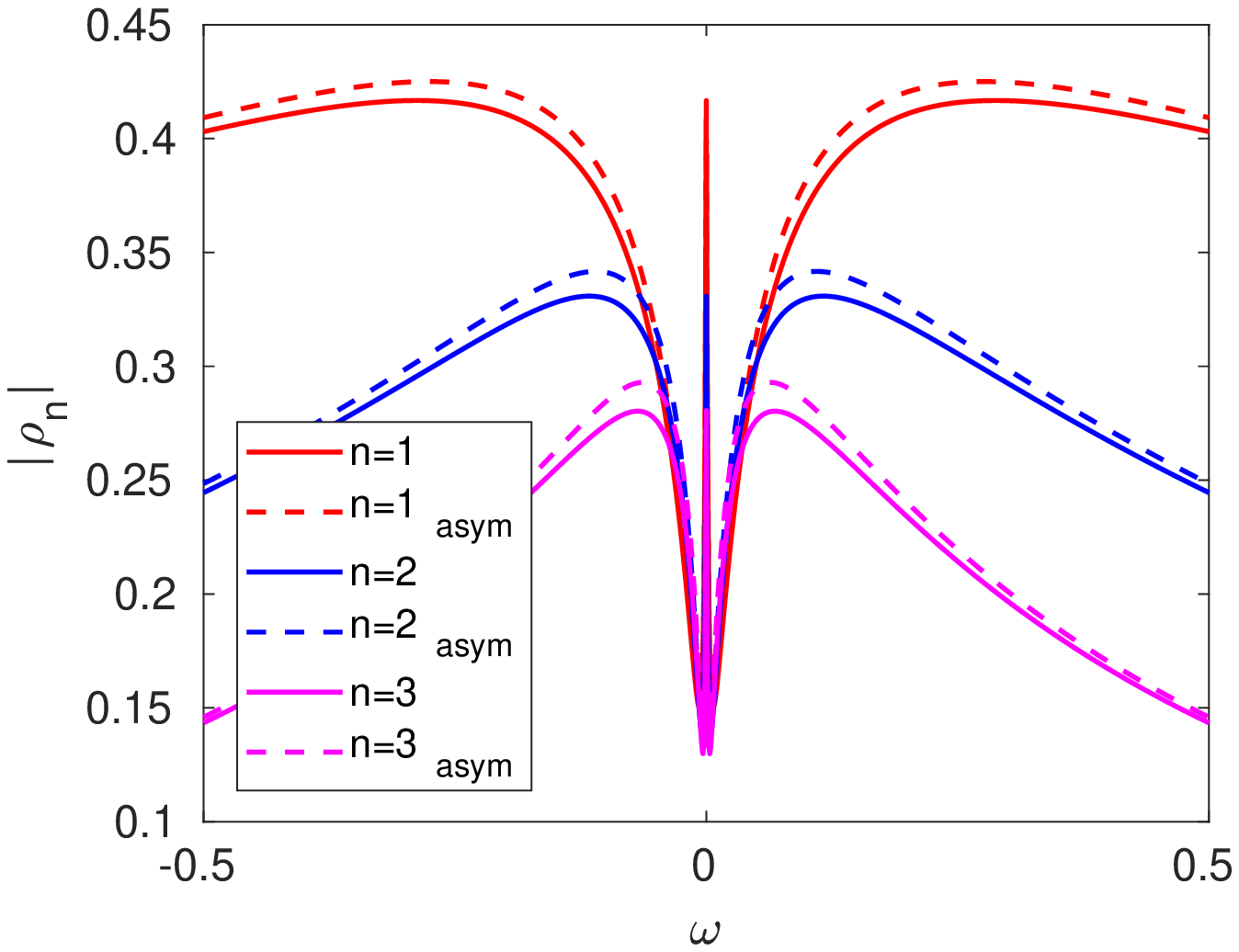}
\caption{Comparison of the optimized $\alpha$ for $\epsilon=10^{-4}$
  for $n=1$ (left) and comparison of the convergence factor for
  different overlaps for OWR (right).}
\label{fig_4_optimal_n}
\end{figure}
  we show the corresponding results for the overlapping case
  $n>0$, and the observations are similar as for $n=0$.

Finally, we consider the case when we split our RC circuit with size
  $N=1000$ into multiple sub-circuits. We choose $\epsilon=10^{-5}$
  and $T=1000$, with zero initial condition, homogeneous source and
  random initial guesses. We use our optimized $\alpha$ in
in Theorem \ref{opt_alpha_n} from the two sub-circuit analysis. 
 We first consider $N_s=5$ and then $N_s=50$ sub-circuits. We see in
  Fig. \ref{fig_7_multiple}
\begin{figure}
\includegraphics[width=0.49\textwidth,clip]{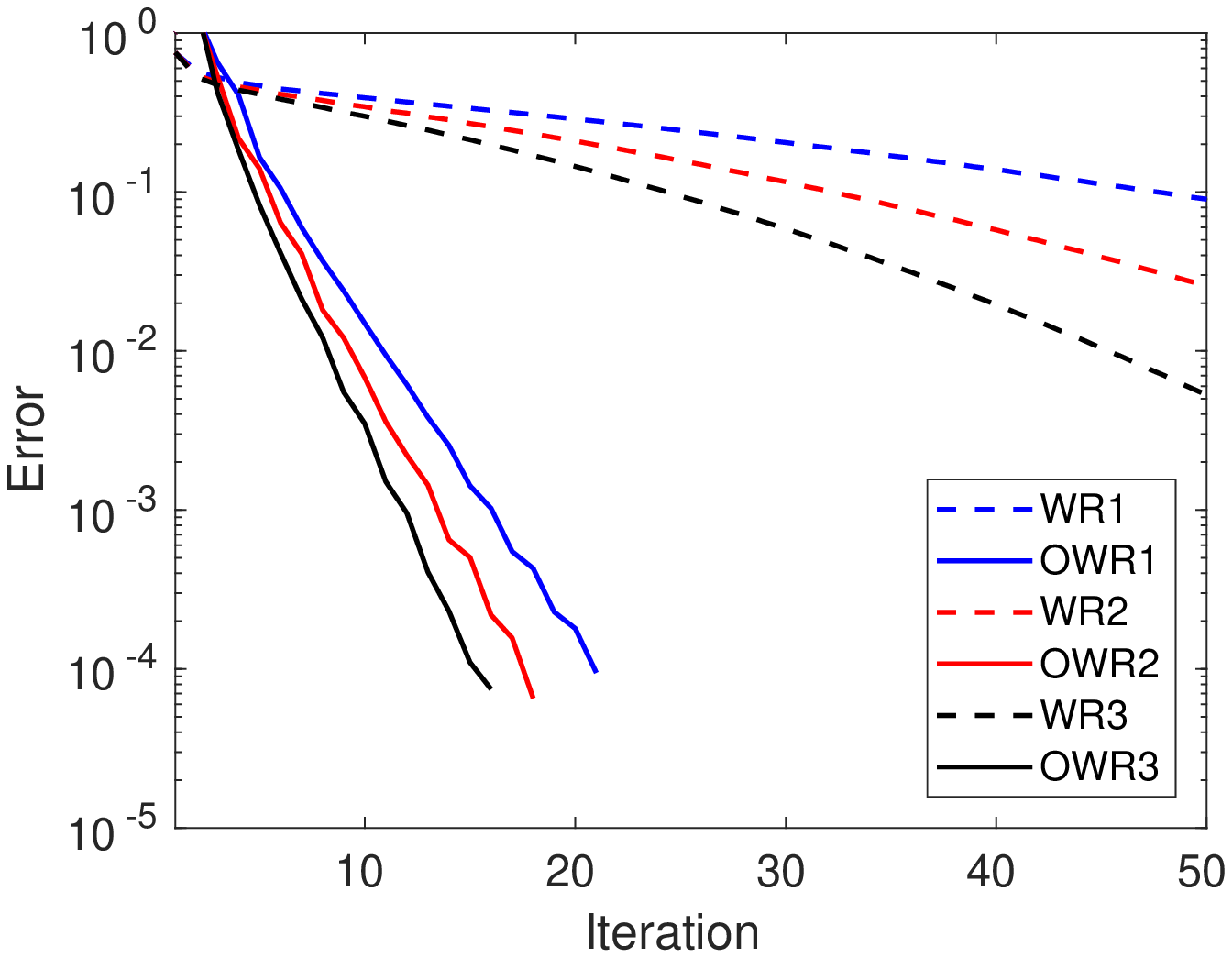}
\includegraphics[width=0.49\textwidth,clip]{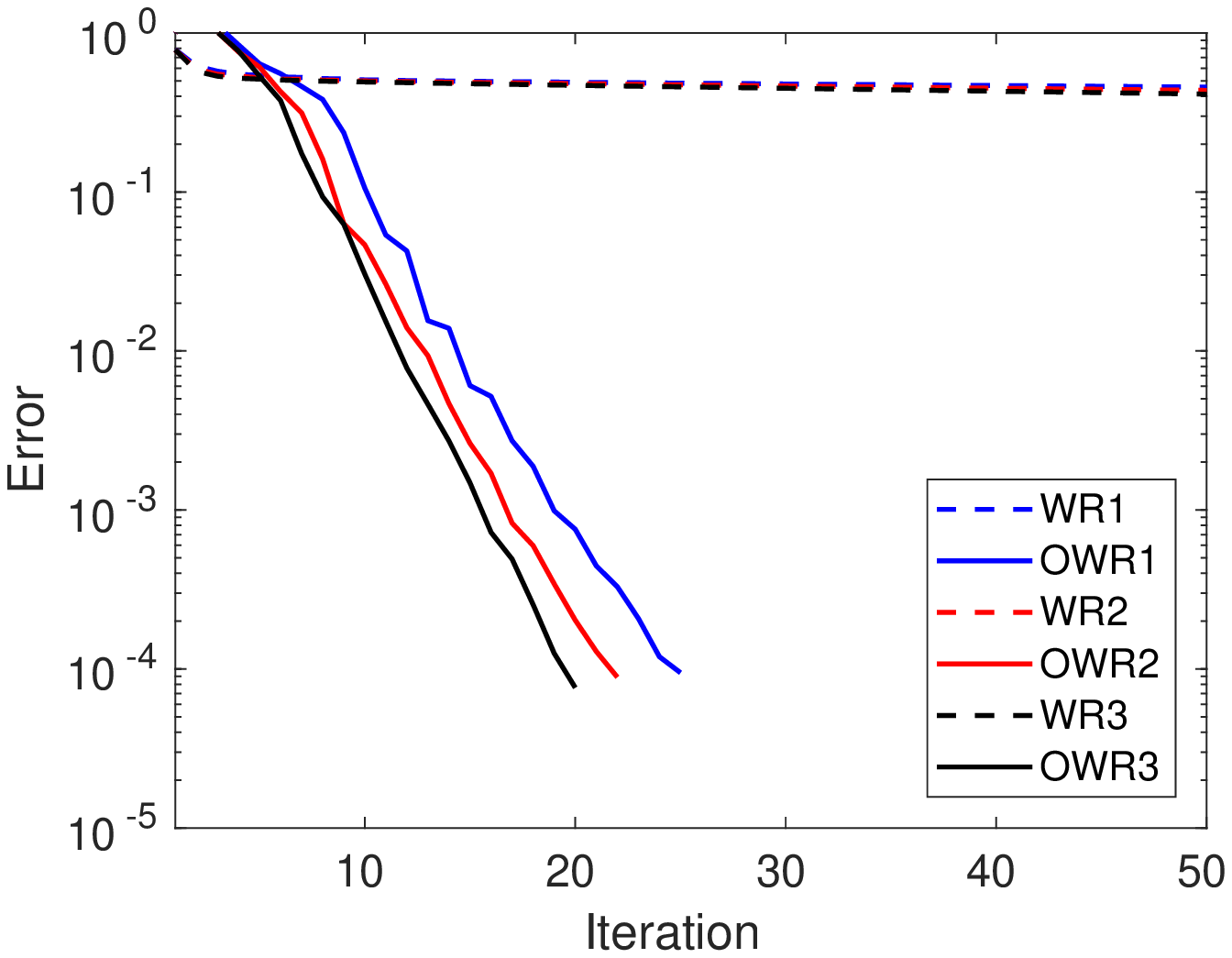}
\caption{Convergence for $T=1000$ for multiple sub-circuits $N_s=5$
  (left) and for $N_s=50$ (right).}
\label{fig_7_multiple}
\end{figure}
  that our asymptotically optimized parameters for two sub-circuits
  work extremely well for the many sub-circuit case, and make the OWR
  solver into a rather effective method for solving many sub-circuit
  problems.
  
  \section{Conclusion}\label{Sec:Conclusion}

  We presented a first analysis for the influence of overlap
  on the convergence factor for both the WR and OWR methods applied
  to RC circuits. We found that increasing the overlap increases
  the convergence rate for both WR and OWR, but the impact of
  optimized transmission conditions is far more important than the
  impact of overlap. The overlap however changes the optimized
  parameters, and we provided closed form asymptotic formulas for
  them. Using OWR with these parameters leads to much lower iteration
  counts, also when OWR is used for many sub-circuits. We also showed
  that our optimized parameters can be used when solving heat
  equations with waveform relaxation techniques. A final important
  contribution is our interpretation of these optimized transmission
  conditions as circuit elements, which should help circuit designers
  to better understand and embrace this new technology.


%
%

\bibliographystyle{spbasic}      
\bibliography{ref}

\begin{thebibliography}{34}
\providecommand{\natexlab}[1]{#1}
\providecommand{\url}[1]{{#1}}
\providecommand{\urlprefix}{URL }
\expandafter\ifx\csname urlstyle\endcsname\relax
  \providecommand{\doi}[1]{DOI~\discretionary{}{}{}#1}\else
  \providecommand{\doi}{DOI~\discretionary{}{}{}\begingroup
  \urlstyle{rm}\Url}\fi
\providecommand{\eprint}[2][]{\url{#2}}

\bibitem[{{Al-Khaleel} et~al.(2013){Al-Khaleel}, {Gander}, and
  {Ruehli}}]{RLCG_KG_OWR_transmissionlines1}
{Al-Khaleel} M, {Gander} MJ, {Ruehli} AE (2013) Optimized waveform relaxation
  solution of {RLCG} transmission line type circuits. In: 9th International
  Conference on Innovations in Information Technology (IIT), pp 136--140

\bibitem[{Al-Khaleel et~al.(2014{\natexlab{a}})Al-Khaleel, Gander, and
  Ruehli}]{RC_small_KG}
Al-Khaleel M, Gander MJ, Ruehli AE (2014{\natexlab{a}}) A mathematical analysis
  of optimized waveform relaxation for a small {RC} circuit. Applied Numerical
  Mathematics 75:61--76, 10th IMACS International Symposium on Iterative
  Methods in Scientific Computing

\bibitem[{Al-Khaleel et~al.(2014{\natexlab{b}})Al-Khaleel, Gander, and
  Ruehli}]{RC_KG}
Al-Khaleel M, Gander MJ, Ruehli AE (2014{\natexlab{b}}) Optimization of
  transmission conditions in waveform relaxation techniques for {RC} circuits.
  SIAM Journal on Numerical Analysis 52

\bibitem[{{Bortot} et~al.(2018){Bortot}, {Auchmann}, {Garcia}, {Navarro},
  {Maciejewski}, {Mentink}, {Prioli}, {Ravaioli}, {Sch\"{o}ps}, and
  {Verweij}}]{FC_steam}
{Bortot} L, {Auchmann} B, {Garcia} IC, {Navarro} AMF, {Maciejewski} M,
  {Mentink} M, {Prioli} M, {Ravaioli} E, {Sch\"{o}ps} S, {Verweij} AP (2018)
  {STEAM}: {A} hierarchical cosimulation framework for superconducting
  accelerator magnet circuits. IEEE Transactions on Applied Superconductivity
  28(3):1--6, \doi{10.1109/TASC.2017.2787665}

\bibitem[{Brenner and Scott(2007)}]{book_FEM}
Brenner S, Scott R (2007) The Mathematical Theory of Finite Element Methods.
  Texts in Applied Mathematics, Springer New York

\bibitem[{Christlieb et~al.(2010)Christlieb, Macdonald, and Ong}]{time_RIDC}
Christlieb AJ, Macdonald CB, Ong BW (2010) Parallel high-order integrators.
  SIAM J Scientific Computing 32:818--835

\bibitem[{{de Dieu Nshimiyimana} et~al.(2016){de Dieu Nshimiyimana}, {Plumier},
  {Dular}, {Geuzaine}, and {Gyselinck}}]{WR_FC_Cosimulation}
{de Dieu Nshimiyimana} J, {Plumier} F, {Dular} P, {Geuzaine} C, {Gyselinck} J
  (2016) Co-simulation of finite element and circuit solvers using optimized
  waveform relaxation. In: IEEE International Energy Conference (ENERGYCON), pp
  1--6, \doi{10.1109/ENERGYCON.2016.7513958}

\bibitem[{Dolean et~al.(2015)Dolean, Jolivet, and Nataf}]{Nataf}
Dolean V, Jolivet P, Nataf F (2015) An introduction to domain decomposition
  methods. Society for Industrial and Applied Mathematics (SIAM), Philadelphia,
  PA, algorithms, theory, and parallel implementation

\bibitem[{Gander and Halpern(2003)}]{heat_GL}
Gander MJ, Halpern L (2003) M\'ethodes de relaxation d'ondes {(SWR)} pour
  l'\`equation de la chaleur en dimension 1. Comptes Rendus Mathematique
  336(6):519 -- 524

\bibitem[{Gander and Kwok(2018)}]{FEM_book}
Gander MJ, Kwok F (2018) Numerical Analysis of Partial Differential Equations
  Using Maple and MATLAB. Society for Industrial and Applied Mathematics,
  Philadelphia, PA, \doi{10.1137/1.9781611975314}

\bibitem[{Gander and Ruehli(2003)}]{RLCG_GA_OWR_line}
Gander MJ, Ruehli AE (2003) Solution of large transmission line type circuits
  using a new optimized waveform relaxation partitioning. In: IEEE Symposium on
  Electromagnetic Compatibility. Symposium Record (Cat. No.03CH37446), vol~2,
  pp 636--641 vol.2

\bibitem[{Gander and Ruehli(2004)}]{RC_small_GA}
Gander MJ, Ruehli AE (2004) Optimized waveform relaxation methods for {RC} type
  circuits. IEEE Transactions on Circuits and Systems I: Regular Papers
  51(4):755--768

\bibitem[{Gander and Stuart(1998)}]{heat_GS}
Gander MJ, Stuart A (1998) Space-time continuous analysis of waveform
  relaxation for the heat equation. SIAM Journal on Scientific Computing
  19(6):2014--2031

\bibitem[{Gander et~al.(2003)Gander, Halpern, and Nataf}]{wave_GLF_optimal1d}
Gander MJ, Halpern L, Nataf F (2003) Optimal {S}chwarz waveform relaxation for
  the one dimensional wave equation. SIAM Journal on Numerical Analysis
  41(5):1643--1681

\bibitem[{Gander et~al.(2006)Gander, Al-Khaleel, and
  Ruehli}]{RLCG_KG_WR_transmissionlines}
Gander MJ, Al-Khaleel M, Ruehli AE (2006) Waveform relaxation technique for
  longitudinal partitioning of transmission lines. In: IEEE Electrical
  Performane of Electronic Packaging, pp 207--210

\bibitem[{Gander et~al.(2009)Gander, Al-Khaleel, and
  Ruchli}]{RLCG_KG_OWR_transmissionlines}
Gander MJ, Al-Khaleel M, Ruchli AE (2009) Optimized waveform relaxation methods
  for longitudinal partitioning of transmission lines. IEEE Transactions on
  Circuits and Systems I: Regular Papers 56(8):1732--1743

\bibitem[{Gander et~al.(2013)Gander, Jiang, and Li}]{time_parareal}
Gander MJ, Jiang YL, Li RJ (2013) Parareal {S}chwarz waveform relaxation
  methods. In: Bank R, Holst M, Widlund O, Xu J (eds) Domain Decomposition
  Methods in Science and Engineering XX, Springer Berlin Heidelberg, Berlin,
  Heidelberg, pp 451--458

\bibitem[{Gander et~al.(2018)Gander, Kumbhar, and Ruehli}]{RC_GK24}
Gander MJ, Kumbhar PM, Ruehli AE (2018) Analysis of overlap in waveform
  relaxation methods for {RC} circuits. In: Domain Decomposition Methods in
  Science and Engineering XXIV, Springer International Publishing, Cham, pp
  281--289

\bibitem[{Gander et~al.(2019{\natexlab{a}})Gander, Jiang, and
  Song}]{time_parareal2}
Gander MJ, Jiang YL, Song B (2019{\natexlab{a}}) A superlinear convergence
  estimate for the parareal {S}chwarz waveform relaxation algorithm. SIAM
  Journal on Scientific Computing 41(2):A1148--A1169

\bibitem[{Gander et~al.(2019{\natexlab{b}})Gander, Kumbhar, and
  Ruehli}]{RLCG_GK25}
Gander MJ, Kumbhar PM, Ruehli AE (2019{\natexlab{b}}) Asymptotic analysis for
  different partitionings of {RLC} transmission lines. In: accepted for Domain
  Decomposition Methods in Science and Engineering XXV, LNCSE, Springer-Verlag

\bibitem[{{Garcia} et~al.(2017){Garcia}, {Sch\"{o}ps}, {Maciejewski}, {Bortot},
  {Prioli}, {Auchmann}, and {Verweij}}]{GS_FC}
{Garcia} IC, {Sch\"{o}ps} S, {Maciejewski} M, {Bortot} L, {Prioli} M,
  {Auchmann} B, {Verweij} A (2017) Optimized field/circuit coupling for the
  simulation of quenches in superconducting magnets. IEEE Journal on Multiscale
  and Multiphysics Computational Techniques 2:97--104,
  \doi{10.1109/JMMCT.2017.2710128}

\bibitem[{Gaspar and Rodrigo(2017)}]{multigrid_heat}
Gaspar FJ, Rodrigo C (2017) Multigrid waveform relaxation for the
  time-fractional heat equation. SIAM Journal on Scientific Computing
  39(4):A1201--A1224

\bibitem[{Griffiths(1986)}]{Griffiths1986}
Griffiths DF (1986) Finite Element Methods for Time Dependent Problems,
  Springer Netherlands, Dordrecht, pp 327--357

\bibitem[{{Ho} et~al.(1975){Ho}, {Ruehli}, and {Brennan}}]{MNA}
{Ho} CW, {Ruehli} A, {Brennan} P (1975) The modified nodal approach to network
  analysis. IEEE Transactions on Circuits and Systems 22(6):504--509

\bibitem[{Kwok and Ong(2019)}]{time_WRAP}
Kwok F, Ong B (2019) {S}chwarz waveform relaxation with adaptive pipelining.
  SIAM Journal on Scientific Computing 41(1):A339--A364

\bibitem[{Lelarasmee et~al.(1982)Lelarasmee, Ruehli, and
  Sangiovanni-Vincentelli}]{WR_LRS}
Lelarasmee E, Ruehli AE, Sangiovanni-Vincentelli AL (1982) The waveform
  relaxation method for time-domain analysis of large scale integrated
  circuits. IEEE Transactions on Computer-Aided Design of Integrated Circuits
  and Systems 1(3):131--145

\bibitem[{Menkad and Dounavis(2017)}]{coupling_WR_tarik}
Menkad T, Dounavis A (2017) Resistive coupling-based waveform relaxation
  algorithm for analysis of interconnect circuits. IEEE Transactions on
  Circuits and Systems I: Regular Papers 64(7):1877--1890

\bibitem[{Najm(2010)}]{book_circuit}
Najm FN (2010) Circuit Simulation. Wiley-IEEE Press

\bibitem[{Ong et~al.(2016{\natexlab{a}})Ong, Haynes, and Ladd}]{RIDC_ong}
Ong BW, Haynes RD, Ladd K (2016{\natexlab{a}}) Algorithm 965: Ridc methods: A
  family of parallel time integrators. ACM Trans Math Softw 43(1):8:1--8:13,
  \doi{10.1145/2964377}

\bibitem[{Ong et~al.(2016{\natexlab{b}})Ong, High, and Kwok}]{time_pipeline}
Ong BW, High S, Kwok F (2016{\natexlab{b}}) Pipeline {S}chwarz waveform
  relaxation. In: Dickopf T, Gander MJ, Halpern L, Krause R, Pavarino LF (eds)
  Domain Decomposition Methods in Science and Engineering XXII, Springer
  International Publishing, Cham, pp 363--370

\bibitem[{Ruehli and Zukowski(1994)}]{WR_RC_Convergence}
Ruehli AE, Zukowski CA (1994) Convergence of waveform relaxation for {RC}
  circuits. In: Coughran WM, Cole J, Llyod P, White JK (eds) Semiconductors,
  Springer New York, New York, NY, pp 141--146

\bibitem[{Sand and Burrage(1998)}]{wr_jacobi}
Sand J, Burrage K (1998) A {J}acobi waveform relaxation method for {ODE}s. SIAM
  Journal on Scientific Computing 20(2):534--552,
  \doi{10.1137/S1064827596306562}

\bibitem[{Vandewalle and Piessens(1991)}]{multigridWR}
Vandewalle S, Piessens R (1991) Multigrid Waveform Relaxation for Solving
  Parabolic Partial Differential Equations, Birkh{\"a}user Basel, Basel, pp
  377--388

\bibitem[{{Wu, Shu-Lin} and {Al-Khaleel, Mohammad D.}(2017)}]{RC_discrete_SM}
{Wu, Shu-Lin}, {Al-Khaleel, Mohammad D} (2017) Optimized waveform relaxation
  methods for {RC} circuits: discrete case. ESAIM: M2AN 51(1):209--223

\end{thebibliography}



\end{document}